\documentclass{amsart}

\usepackage{graphicx, color}
\usepackage{amssymb, amsmath, amsthm}
\usepackage{epstopdf}
\usepackage{verbatim}

\newtheorem{theorem}{Theorem}[section]
\newtheorem{lemma}[theorem]{Lemma}
\newtheorem{proposition}[theorem]{Proposition}
\newtheorem{corollary}[theorem]{Corollary}

\theoremstyle{definition}
\newtheorem{definition}[theorem]{Definition}
\newtheorem{notation}[theorem]{Notation}

\newtheorem{assumption}{Assumption}
\theoremstyle{remark}
\newtheorem{remark}[theorem]{Remark}

\numberwithin{equation}{section}

\newcommand{\Z}{\mathbb{Z}}

\newcommand{\R}{\mathbb{R}}

\newcommand{\size}{\textrm{size}}
\newcommand{\sizee}{\textrm{\emph{size}}}
\newcommand{\energy}{\textrm{energy}}
\newcommand{\energye}{\textrm{\emph{energy}}}

\begin{document}

\title{On a biparameter maximal multilinear operator}

\author{Peter M. Luthy}

\address{Department of Mathematics, Washington University in St. Louis, MO 63130}

\email{luthy@math.wustl.edu}

\thanks{This work was completed by the author while he was a graduate student at Cornell University.}


\begin{abstract}
It is well-known that estimates for maximal operators and questions of pointwise convergence are strongly connected.  In recent years, convergence properties of so-called `non-conventional ergodic averages' have been studied by a number of authors, including Assani, Austin, Host, Kra, Tao, and so on.  In particular, much is known regarding convergence in $L^2$ of these averages, but little is known about pointwise convergence. In this spirit, we consider the pointwise convergence of a particular ergodic average and study the corresponding maximal trilinear operator (over $\mathbb{R}$, thanks to a transference principle).  Lacey in \cite{LBMO} and Demeter, Tao, and Thiele in \cite{DTT} have studied maximal multilinear operators previously; however, the maximal operator we develop has a novel bi-parameter structure which has not been previously encountered and cannot be estimated using their techniques.  We will carve this bi-parameter maximal multilinear operator using a certain Taylor series and produce non-trivial H\"{o}lder-type estimates for one of the two ``main'' terms by treating it as a singular integrals whose symbol's singular set is similar to that of the Biest operator, studied by Muscalu, Tao, and Thiele in \cite{MTTBiest1} and \cite{MTTBiest2}.
\end{abstract}

\maketitle

\section{Overview and Structure}

In this paper, we begin the study of boundedness properties of the biparameter maximal multilinear operator
\begin{equation}\label{luthyformula1}
(f_1,f_2,f_3)\mapsto\sup_{h_1,h_2}\frac{1}{h_1h_2}\int_{-h_1}^{h_1}\int_{-h_2}^{h_2}|f_1(x-t)f_2(x+s+t)f_3(x-s)|dsdt,
\end{equation}
where the term biparameter refers to the fact that the supremum involves two parameters, $h_1$ and $h_2$.  This operator arises from certain convergence questions in ergodic theory; the connection will be described below.  Similar mono-parameter maximal bilinear operators have been studied by Michael Lacey, \cite{LBMO}; this work was later generalized to mono-parameter $n$-linear operators by Ciprian Demeter, Terence Tao, and Christoph Thiele, \cite{DTT}.  Their work establishes, after much effort, that many mono-parameter maximal multilinear operators, including (\ref{luthyformula1}) with $\sup_{h_1,h_1>0}$ replaced by $\sup_{h_1=h_2>0}$, can be estimated using singular integral techniques related to the bilinear Hilbert transform.  In this paper we give a proof-of-concept that operators of the form (\ref{luthyformula1}) should be estimatable using techniques related to the so-called Biest operator studied by Camil Muscalu, Terence Tao, and Ciprian Demeter, \cite{MTTBiest1}, \cite{MTTBiest2}.  The Biest operator is related to the study of a certain class of dynamical systems coming from families of integrable PDE.

In particular, we present a study of (\ref{luthyformula1}) by performing a certain time-frequency discretization process on this operator; the operator (\ref{luthyformula1}) can be carved into two main pieces and, modulo analyzing certain ``error terms'' coming from a Taylor series argument, the present work will show that the simpler of these two pieces can be studied in terms of the following discrete time-frequency model operator: for finite families of rank 1 tri-tiles $\vec{\mathbf{P}}$ and $\vec{\mathbf{Q}}$ and functions $\phi_P^t$ and $\phi_Q^t$ for $t=1,2,3$ which are $L^2$-normalized and adapted to the tri-tiles in the appropriate way, the model is given by
\[
\sum_{P\in\vec{\mathbf{P}}}\frac{1}{|I_P|^{1/2}}\langle f_3, \phi_P^1\rangle\left\langle B_P(f_1,f_2), \phi_P^2\right\rangle\left\langle f_41_{|I_P|\ge 2^{N_2(x)}},\phi_P^3\right\rangle,
\]
where
\[
B_P(f_1,f_2):=\sum_{Q\in\vec{\mathbf{Q}}:\omega_{Q_3}\subset \omega_{P_2}}\frac{1}{|I_Q|^{1/2}}\langle f_1, \phi_{Q}^1\rangle\langle f_2, \phi_{Q}^2\rangle\phi_{Q}^3
\]
and $N_2$ is an arbitrary integer-valued function on $\mathbb{R}$.  The more complicated of these two pieces will be described in later work.  Later sections of the paper will describe the aforementioned discretization process more fully.  Additionally, we will show the following:
\begin{theorem}
The above model operator satisfies the same restricted weak-type estimates as the Biest operator studied in \cite{MTTBiest1}, \cite{MTTBiest2}.  In particular, the above model operator is of restricted weak-type for all 4-tuples $(1/p_1,1/p_2,1/p_3,1/p_4)$ in the interior of the convex hull of the following twelve points in $\mathbb{R}^4$:
\[
\begin{array}{cccccc} 
(1,-\frac{3}{2},\frac{1}{2},1) & (1,-\frac{3}{2},1,\frac{1}{2})&(-\frac{3}{2},1,\frac{1}{2},1) & (-\frac{3}{2},1,1,\frac{1}{2}),\\
\\
(1,0,-\frac{1}{2},\frac{1}{2}) & ( 1,0,-\frac{1}{2},\frac{1}{2}) & ( \frac{1}{2},0,-\frac{1}{2},1) & (0,\frac{1}{2},-\frac{1}{2},1),  \\
\\
(1,0,\frac{1}{2},-\frac{1}{2}) & (0,1,\frac{1}{2},-\frac{1}{2}) & (\frac{1}{2},0,1,-\frac{1}{2}) & (0,\frac{1}{2},1,-\frac{1}{2}).
\end{array}
\]
One such interior point is given by $(1/2,1/2,1/2,-1/2)$.  This causes a simplified variant of (\ref{luthyformula1}) to satisfy a strong $L^2\times L^2\times L^2\rightarrow L^{2/3}$ estimate which cannot be produced using H\"{o}lder's inequality and known estimates. 
\end{theorem}

This theorem will be the main ingredient in proving the main result, Theorem \ref{mainresult}.\\

Section 2 will provide motivation for the study of (\ref{luthyformula1}) and provide some context for the operators mentioned in the preceding paragraph.  It is hoped that this section will be readable by a fairly broad audience.  In Section 3 we carefully discuss how to produce the above model operator from (\ref{luthyformula1}) and point out explicitly the relationship to and differences from the Biest operator; we will frequently pause to give heuristic explanations before providing rigorous proofs.  In Section 4, we describe restricted weak-type interpolation which is a vital theorem in the analysis of multilinear operators when the target space is an $L^p$ space for $p<1$.  In Section 5 we provide size and energy estimates for the above model operator.  In Section 6 we describe how the size and energy estimates give the full range of estimates coming from the Biest operator.  In Section 7, we provide the main result.

\section{Motivation and Context}
\subsection{Pointwise Convergence and Maximal Operators}
For sake of exposition, we begin with some very well known results connecting pointwise convergence and estimates for maximal operators.\\

Given a sequence of functions, there is a variety of ways the sequence might converge: pointwise, in norm, weakly, and so on.  Pointwise convergence is na\"{i}vely the most ``natural'' but is difficult to work with in the framework of modern analysis.  With this in mind, we recall two classical theorems.
\begin{theorem}[Lebesgue Differentiation Theorem]\label{LDT}
If $B_r(x)$ denotes the ball of radius $r$ around $x$ in $\mathbb{R}^d$, then given $f\in L^p(\mathbb{R}^d)$ for $p\ge 1$, the average value of $f$ on $B_r(x)$ converges for a.e. $x$ to $f(x)$ as $r\rightarrow 0$.
\end{theorem}
\begin{theorem}[Carleson--Hunt]\label{CHT}
If $\mathbb{T}$ denotes the unit circle and $f\in L^p(\mathbb{T})$ for $p>1$, then the (symmetric) partial sums of the Fourier series for $f$ converge pointwise to $f$ almost everywhere.
\end{theorem}

The proofs of the two theorems are very different, but both come down to proving the theorem for a dense class of functions and that a maximal operator is bounded.  Smooth functions serve as suitable dense function classes for both theorems.  For Theorem \ref{LDT}, the relevant maximal operator is the well-known Hardy--Littlewood maximal operator $M$,
\[
M(f)(x)=\sup_{r>0}\frac{1}{m(B_r)}\int_{B_r}|f(x+t)|dt,
\]
which is bounded from $L^p\rightarrow L^p$ for $p\in (1,\infty]$ and bounded from $L^1$ to weak-$L^1$;  for Theorem \ref{CHT}, this is the Carleson operator $C$,
\[
C(f)(x)=\sup_{N\in\mathbb{R}}\left|\int_{-\infty}^N\hat{f}(\xi)e^{2\pi ix\xi}d\xi\right|,
\]
which is bounded from $L^p$ to $L^p$ for $p\in(1,\infty)$.  That the Carleson-Hunt theorem is false for $p=1$ is a result of Kolmogorov from the 1920s and is reflected in the fact that $C$ does not satisfy a suitable $L^1$ estimate.  Hence proving estimates for maximal operators seems to be a main ingredient in proving pointwise convergence theorems.  The following partial converse of Stein says that this is fundamentally true:
\begin{theorem}[Stein, 1961\footnote{This theorem is true in much greater generality, but the requirement that $p\le 2$ cannot be dropped, in general.  See \cite[Theorem 1]{stein} for the exact statement.}]
Suppose that $T_n$ is a family of bounded linear operators on $L^p(\mathbb{T})$ for $p\in [1,2]$ which commute with translations (i.e. rotations of the circle).  Further, suppose that for each $f\in L^p$ and almost every $x$, $T_n(f)(x)$ converges pointwise.  Then the operator $f\mapsto \sup_n |T_nf|$ is bounded from $L^p$ to weak-$L^{p}$.
\end{theorem}
Thus there is, to a certain extent, an equivalence of pointwise convergence and boundedness of certain operators, at least in the linear setting.

\subsection{Pointwise Convergence in Ergodic Theory}
We begin with the following standard definition.
\begin{definition}[Ergodic Transformation]
Let $(X,\sigma,p)$ be a complete probability space and $T:X\rightarrow X$ be an invertible, bimeasurable map which preserves measure, i.e. $pT^{-1}(E)=p(E)$.  Let $\mathcal{I}$ denote the collection of sets $E$ with $T^{-1}(E)=E$.  $\mathcal{I}$ is called the invariant sigma algebra of $T$.  If $\mathcal{I}$ is the trivial sigma algebra (i.e. every element of $\mathcal{I}$ has probability $1$ or $0$) we say that $T$ is ergodic.
\end{definition}

Let $(X,\sigma,p)$ be a complete probability space and suppose that $T:X\rightarrow X$ is an invertible, bimeasurable map which preserves measure.  If $f\in L^p(X),$ the following equality holds almost everywhere:
\[
\lim_{n\rightarrow \infty}\frac{1}{n}\sum_{k=1}^n f(T^kx)=E(f|\mathcal{I}),
\]
where $E(f|\mathcal{I})$ is the conditional expectation of $f$ with respect to the invariant $\sigma$-algebra of $T$.  If $T$ is an ergodic transformation, $\mathcal{I}$ is trivial, and so the right side is actually $\int_Xf$.  In this case, the above equality is the celebrated Birkhoff Ergodic Theorem.  

Limits of ergodic averages in the spirit of Birkhoff's theorem have been studied by many authors with a host of applications in mathematics as well as the natural sciences.  One of the heralded applications of ergodic theory is Furstenberg's proof of Szemer\'{e}di's theorem:
\begin{theorem}[Szemer\'{e}di's Theorem]
Any subset of the natural numbers having positive upper density\footnote{Here, the upper density of a subset $E$ of the integers is $\lim_{n\rightarrow\infty} \left|[1,n]\cap E\right|/n$.} contains arithmetic sequences of arbitrary length.\footnote{Of course this theorem was recently extended to the set of primes by Green and Tao in \cite{taogreen}.  This required different methods since the primes do not have positive upper density --- by the Prime Number Theorem, the relevant quantity for upper density decays like $1/\log n$.}
\end{theorem}
The main ingredient in Furstenberg's proof is:
\begin{theorem}[Furstenberg's Multiple Recurrence Theorem]
Let $(X,\sigma,p)$ be a probability space and $T$ as in Birkhoff's theorem.  If $E$ has positive measure, then for any $k>0$ there exists an $n$ such that
\[
p(E\cap T^{-n}E\cap T^{-2n}E\cap...\cap T^{-kn}E)>0.
\]
\end{theorem}
Insofar as Szemer\'{e}di's theorem is concerned one should think of $T$ as $T(x)=x+1$, so that the positivity of the above probability guarantees that $E$ contains some arithmetic sequence of length $k$.  This is not exactly correct --- the upper density is not a probability on $\mathbb{Z}$, for instance --- but Furstenberg was able to avoid this technical difficulty.  Although Furstenberg's proof avoids the issue, it would be nice if Birkhoff's theorem extended to sequences such as
\[
\frac{1}{n}\sum_{k=1}^nf_1(T^kx)f_2(T^{2k}x)...f_m(T^{mk}x)
\]
converging pointwise to something positive for any $m$ (here one should think that, for all $i$, $f_i=1_E$ for some fixed set of positive upper density).  For $m=1$, this is Birkhoff's theorem.  The case $m=2$ was established for $f_1,f_2\in L^\infty$ by Bourgain, \cite{Bourgain}, more than twenty years ago.  However even for $m=3$, the question of pointwise convergence of such averages remains open.\footnote{If one treats the related maximal trilinear operator as a singular integral operator using the methods we will discuss later on, then the related singular integral operator is ``morally'' the trilinear Hilbert transform, for which no estimates are known.}  Recent work of Austin, \cite{Austin}, establishes, among much more general types of averages, that
\[
\frac{1}{|I_N|}\sum_{n\in I_N+a_N}\prod_{i=1}^df_i(T^{in}x)
\]
converges in $L^2$-norm to some function whenever $I_N$ is some F{\o}lner sequence of subsets of integers --- this work generalizes a variety of papers by other authors, e.g. Tao, \cite{Tao}, Host and Kra, \cite{hostkra}, and Ziegler, \cite{Ziegler}.  In the work of Furstenberg and Weiss, \cite{FW}, expressions like
\[
\frac{1}{N}\sum_{n=1}^Nf_1(T^{n}x)f_2(T^{n^2}x)
\]
are also shown to converge in $L^2$.  More complicated averages involving $k$ independent parameters in the sum and $2^{k}-1$ functions, such as
\[
\frac{1}{N^3}\sum_{n,m,p=0}^Nf_1(T^nx)f_2(T^mx)f_3(T^px)f_4(T^{n+m}x)f_5(T^{n+p}x)f_6(T^{m+p}x)f_7(T^{n+m+p}x),
\]
are shown to converge almost everywhere by Assani \cite{Assani}.

This large body of work suggested a natural extension, namely whether the bi-parameter average
\begin{equation}\label{biaverage}
\frac{1}{2M+1}\frac{1}{2N+1}\sum_{m=-M}^M\sum_{n=-N}^Nf_1(T^mx)f_2(T^{-m-n}x)f_3(T^nx)
\end{equation}
converges pointwise almost everywhere, where $M$ and $N$ go to infinity at different rates.  As discussed in the previous section, questions of pointwise convergence are deeply related to boundedness of maximal operators.  Rather than work in the generality of a dynamical system, one can use a correspondence principle to translate the problem to $\mathbb{R}$.  For example, see Section 14 of \cite{DTT}.  The maximal operator one produces via such a correspondence principle is precisely
\begin{equation}\label{luthyformula}
(f_1,f_2,f_3)\mapsto\sup_{h_1,h_2}\frac{1}{h_1h_2}\int_{-h_1}^{h_1}\int_{-h_2}^{h_2}|f_1(x-t)f_2(x+s+t)f_3(x-s)|dsdt.
\end{equation}
Forcing $h_1:=h_2$, one obtains, essentially, the object of the main result in \cite{DTT} by Demeter, Tao, and Thiele.  However, the above maximal operator depends on two independent parameters, $h_1$ and $h_2$, and so we call it a \emph{bi-parameter maximal operator}.  In what follows, we discuss this operator in detail; this is, as far as the author knows, the first time such an operator has been studied.

In particular, we will show that an operator related to (\ref{luthyformula}) is bounded from $L^{p_1}\times L^{p_2}\times L^{p_3}\rightarrow L^{p_4}$, for $p_1,p_2,p_3>1$ with $1/p_1+1/p_2+1/p_3=1/p_4$, for a ``non-trivial'' range of exponents $p_i$.\footnote{Clearly, one expects a H\"{o}lder-type condition on the exponents since this operator behaves like a pointwise product for a fixed pair $h_1,h_2$.}  The term ``non-trivial'' here requires some explanation.  One could, for example, assume that $f_2\in L^\infty$ in which case (\ref{luthyformula}) splits into a tensor product  of Hardy--Littlewood operators and thus H\"{o}lder's inequality and well-known results produce ``trivial'' estimates.  However, one would ideally like all the $f_i$ to be as close to $L^1$ as possible, in which case a number of things go awry.  Indeed, in such a case, the target space $L^{p_4'}$ has $p_4'<1$ and $p_4<0$, in which case the triangle inequality no longer holds, the relationship between an operator and its adjoint is more complicated, and the 4-linear form one produces by dualizing cannot support H\"{o}lder's inequality.  Alternatively, one could put $f_3\in L^\infty$ and invoke other known results --- this produces, essentially, a maximal variant of  $B(f_1,M(f_2))$, where $M$ is the Hardy--Littlewood maximal operator and $B$ is the bilinear Hilbert transform, an operator which can be handled by the techniques of \cite{LBMO} and \cite{DTT}.  There are a number of such possible trivial estimates which are available.  One can then invoke multilinear interpolation results to produce a large family of estimates which require only known results.  In this article, we produce results outside these easily available estimates to push the range of allowable exponents even further.

\subsection{Connection to Singular Integral Operators}
Returning to the boundedness of the Hardy-Littlewood maximal operator, we recall that the proof depends on a classical Vitali covering argument; in particular it does not require any Fourier analysis.  However, the proof does not extend to the bilinear variant,
\begin{equation}\label{bimaximal}
\sup_{r>0}\frac{1}{2r}\int_{-r}^r|f(x+t)g(x+2t)|dt.
\end{equation}
This maximal operator corresponds to the pointwise convergence problem of Bourgain described above (modulo some details).  Of course, one has immediate estimates for the above expression via H\"{o}lder's inequality, but one would like, for example, to have both $f$ and $g$   close to $L^1$, which cannot be handled by H\"{o}lder.  One can, however, use techniques from singular integrals to get estimates outside the usual H\"{o}lder range.  For instance, it was known for a long time that Littlewood-Paley theory could be used to prove the boundedness of the Hardy-Littlewood operator, even though such sophistication was not necessary.  More recently, Lacey, \cite{LBMO}, estimated the above maximal operator using methods related to estimating a maximal variant of the bilinear Hilbert transform,
\[
\sup_{h>0}\int_{h<|t|<1/h}f(x+t)g(x+2t)\frac{dt}{t}.
\]
The results of work by Demeter, Tao, and Thiele, \cite{DTT}, extended this idea to one-parameter maximal $n$-linear operators by realizing that the $n$-linear problem is treatable using the techniques from the maximal \emph{bilinear} Hilbert transform.

The main results of this article center on extending the ideas of Lacey and Demeter, Tao, and Thiele to bi-parameter maximal operators.  In particular, if the Demeter-Tao-Thiele theorem, \cite{DTT}, shows a connection between maximal one-parameter multilinear operators and the maximal bilinear Hilbert transform, the main theorem we prove at present establishes a connection between bi-parameter maximal multilinear operators and a maximal variant of the so-called Biest operator (see \cite{MTTBiest1},\cite{MTTBiest2}) which is connected to AKNS systems --- these systems are a way of describing many integrable PDE.

\subsection{Biest and AKNS systems}
It has been known for some time that there is a strong connection between PDE and time-frequency analysis based on the Heisenberg principle, e.g. as discussed by C. Fefferman in \cite{Feff}; in this paragraph we describe a relevant example which inspired the development of the aforementioned Biest operator.  In \cite{CK1} and \cite{CK2}, Christ and Kiselev were interested in proving that eigenfunctions of one-dimensional Schr\"{o}dinger operators with potential $F$ in $L^p$ are bounded for almost all energies when $p<2$; in their proof, they produced a collection of multilinear operators $T_n$ and wrote eigenfunctions as a sum of multilinear operators $\sum_n T_n(F,...,F)$.  Their methods broke down when the input functions were all in $L^2$, although it was conjectured that eigenfunctions would be bounded when $p=2$.  Muscalu, Tao, and Thiele, using time-frequency analysis, showed that some of these multilinear operators were in fact unbounded when the input functions are in $L^2$ in \cite{MTTCounter}.  This indicates that the multilinear expansion approach is flawed at $p=2$, though the conjecture may still hold --- after all, $e^{ix}$ is a bounded function even though most terms in its power series are not.  One can translate the entire discussion to the framework of the aforementioned AKNS systems, to which many integrable PDEs relate.  One again produces a family of operators, the simplest of which resemble the Carleson operator and the bilinear Hilbert transform; these are important ``protoypical'' objects in time-frequency analysis.  Indeed, Muscalu, Tao, and Thiele studied a variety of operators arising in this way --- the so-called Bi-Carleson, \cite{MTTBiCarleson}, and Biest, \cite{MTTBiest1}, \cite{MTTBiest2}, operators.  Since Muscalu, Tao, and Thiele's approach to the Biest was so fruitful to the present work, we shall present a terse overview of AKNS systems and how they relate to singular integrals.

AKNS\footnote{AKNS systems are named after  M. J. Ablowitz, D. J. Kaup, A. C. Newell, and H. Segur; see, e.g., Chapter 1 in \cite{Darboux}} systems are systems of ODEs capable of describing a wide variety of PDE, such as the KdV, Nonlinear Schr\"{o}dinger, and sine-Gordon equations.  We proceed with a brief discussion of AKNS systems and the connection to singular integrals.  Suppose that $u:=(u_1(t),...,u_n(t))$ is a column vector of complex-valued functions on the line.  Let $D$ be a diagonal $n\times n$ matrix with distinct (constant) entries $d_i$ along the diagonal.  Suppose that $V$ is a matrix whose entries $V_{ij}$ are functions such that diagonal $V_{ii}\equiv 0$.  Let $\lambda$ be a real parameter.  One of the defining equations in an AKNS system is
\[
\frac{d}{dt}u=i\lambda Du+Vu.
\]
The rough (and incorrect) heuristic is that the functions $u_i$ represent the positions in the plane of planets rotating around the origin at rates $d_i$; the $i$th planet affects the motion of the $j$th planet according to the potential $V_{ij}$.  

As a particular example, consider for a fixed function $F$,
\[
\begin{pmatrix} u_1' \\ u_2' \end{pmatrix}=i\lambda \begin{pmatrix} 1 & 0\\ 0 & -1\end{pmatrix}\begin{pmatrix}u_1\\u_2\end{pmatrix}+\begin{pmatrix}0 & 1\\ F & 0\end{pmatrix}\begin{pmatrix} u_1\\u_2\end{pmatrix}.
\]
After doing some algebra, one is easily able to produce the time-independent Schr\"{o}dinger equation,
\[
-u_1''+Fu_1=\lambda^2u_1.
\]

Going back to the general case: supposing that $V$ is upper-triangular, one has, heuristically, that the mass of each planet is vastly bigger than the next\footnote{For instance with the Sun, Jupiter, and Jupiter's moon Io: each is $\approx 1000$ times heavier than the next.}.  After a simple substitution, $u_i(t)=w_i(t)e^{id_it\lambda}$, this equation becomes
\[
w'=Vw
\]
where $w=(w_1,...,w_n)$ and $V=(V_{ij}(t)e^{i\lambda(d_i-d_j)t})$.  In the simplest case, $n=2$ and $V$ an upper triangular matrix, one can solve the system exactly to see that $w_2\equiv C_\lambda$ for some constant $C_\lambda$ and
\[
w_1(t)=C_\lambda\int_{-\infty}^tV_{12}(s)e^{i\lambda(d_1-d_2)s}ds +D_\lambda,
\]
for some constant $D_\lambda$.  Forgetting the constants and assuming for simplicity $d_1-d_2=1$, we see that bounding $\|w_1\|_\infty$ is equivalent to estimating
\[
\sup_{t\in\mathbb{R}}\left|\int_{-\infty}^tV_{12}(s)e^{i\lambda s}ds\right|.
\]
This is trivially finite if $V_{12}\in L^1$.  However, by proving the above expression is $p$-integrable (with respect to $\lambda$) for some $p$, one immediately gets the expression is finite for almost every $\lambda$.  By a theorem of Menshov and Zygmund, this is true for $p\in[1,2)$.  Even further, observe that this expression looks very similar to the Carleson operator described at the beginning of this introduction, except that the integrand has $V_{12}$ rather than the Fourier transform thereof.  If one presumes that $V_{12}$ is the Fourier transform of some function in $L^q$ for $q\in(1,2]$, the boundedness of the Carleson operator, along with the Hausdorff-Young inequality, guarantee boundedness of orbits.  A similar treatment of the $n=3$ upper-triangular case produces a maximal bilinear operator, dubbed the Bi-Carleson operator, studied by Muscalu, Tao, and Thiele, \cite{MTTBiCarleson}. In \cite{MTTBiest1}, \cite{MTTBiest2} Muscalu, Tao, and Thiele studied a non-maximal operator, dubbed the Biest, related to the $n=4$ AKNS.  There are certain structural similarities in the form of the Biest operator and the bi-parameter maximal operator studied in the present work.  After transforming the Biest operator into frequency variables, its symbol has discontinuities along two hyperplanes, $\xi_1=\xi_2$ and $\xi_2=\xi_3$; one then performs a decomposition with respect to this singular set.  It will be convenient to treat our bi-parameter maximal operator in an analogous fashion.

\section{Discretization}
\subsection{Main Problem}

For measurable functions $f_1,f_2,f_3$ with appropriate conditions (to be defined later, but one may assume that these function are smooth, bounded, compactly supported, etc.), our operators, $T$ and $T^*$, are defined by
\[
T(f_1,f_2,f_3)=\frac{1}{h_1h_2}\int_{-h_1}^{h_1}\int_{-h_2}^{h_2}|f_1(x-s)||f_2(x+s+t)||f_3(x-t)|dsdt
\]
and
\[
T^*(f_1,f_2,f_3)=\sup_{h_1,h_2} \frac{1}{h_1h_2}\int_{-h_1}^{h_1}\int_{-h_2}^{h_2}|f_1(x-s)||f_2(x+s+t)||f_3(x-t)|dsdt,
\]
where the supremum is taken over all real $h_1$ and $h_2$.  We wish to show that $T^*$ satisfies H\"{o}lder-type estimates.  Some standard limiting arguments along with restricted weak-type interpolation theorems common in time-frequency analysis will allow us to restrict our attention to smooth functions $f_i$ which are supported on unions of compact intervals such that the $f_i$ have $L^\infty$-norm bounded by $1$.  We will discuss weak-type interpolation later on.  It is often heuristically useful to imagine that the $f_i$ are simply characteristic functions of a union of intervals --- the smoothness condition simply makes the Fourier analysis nicer.

A trivial argument shows that it suffices to modify our operator slightly to include only dyadic values of the $h_j$, i.e. to shift our attention to
\[
T^*(f_1,f_2,f_3)=\sup_{k_1,k_2\in\mathbb{Z}}\frac{1}{2^{k_1}2^{k_2}}\int_{-2^{k_1}}^{2^{k_1}}\int_{-2^{k_2}}^{2^{k_2}}|f_1(x-s)f_2(x+s+t)f_3(x-t)|dsdt
\]
where $k_1,k_2\in\mathbb{Z}$.

\subsection{Fourier Representation}
In the above, we would like to replace the sharp cutoff functions $\chi_{[-2^{k_1},2^{k_1}]}$ and $\chi_{[-2^{k_2},2^{k_2}]}$ with smooth functions; clearly, it would suffice to replace these sharp cutoffs by Schwartz functions $\theta(2^{-k_i}s)$, say, where $\theta$ is non-negative, $1$ at $0$ and which decays rapidly in units of length $1$ away from $[-1,1]$. It may at first glance seem better to pick $\theta$ to be compactly supported, but this results in perfect localization in space variables rather than frequency variables.  Since we should like to use Fourier analysis, it will be more convenient for the Fourier transforms of the functions to be compactly supported.  We will define our functions explicitly via the following lemma.  First, a definition:
\begin{definition}
For smooth functions $\eta_1,\eta_2$, we define $T^*_{\eta_1,\eta_2}$ as follows:
\[
\sup_{k_1, k_2}\frac{1}{2^{k_1}2^{k_2}}\int_{\mathbb{R}^2}|f_1(x-s)f_2(x+s+t)f_3(x-t)|\check{\eta_1}(s/2^{k_1})\check{\eta_2}(t/2^{k_2})dsdt.
\]
\end{definition}
\begin{lemma}
There are symmetric, non-negative, real-valued functions $\alpha$ and $\beta$ which are supported in $[-1,1]$ whose Fourier transforms are non-negative and so that $\check{\alpha}(0)=\check{\beta}(0)=1.$  Moreover,
\[
T^*(f_1,f_2,f_3)(x)\lesssim T^*_{\alpha,\beta}(f_1,f_2,f_3)(x),
\]
where the implied constant depends on the choice of $\alpha$ and $\beta$.
\end{lemma}
\begin{proof}
Let $\theta$ be a nonzero symmetric, real-valued function supported on $[-1/2,1/2]$.  Then $\theta*\theta$ is a real-valued symmetric function supported in $[-1,1]$; since $\theta$ is symmetric, $\hat{\theta}$ is necessarily real-valued so that $(\hat{\theta})^2\ge 0$.  We may then take $\alpha$ and $\beta$ to be $(\theta*\theta)^2$, which will again be symmetric, be supported in $[-1,1]$, and have non-negative Fourier transform (being the convolution of non-negative functions); it is also itself non-negative, being the square of a real-valued function.  We also observe that $\widehat{(\theta*\theta)^2}(0)=\int(\theta*\theta)^2(y)dy>0$, and so we may normalize this function to get $\check{\alpha}(0)=\check{\beta}(0)=1$.  

Since $\alpha,\beta\ge 0$ and $\check{\alpha}(0)=\check{\beta}(0)=1$, we may choose a constant $C$, which depends on our choice of $\theta$, so that $\alpha(x/C)\beta(y/C)$ is pointwise greater than $\frac{1}{2}\chi_R$ where $R$ is the rectangle $[-1,1]\times[-1,1]$, which gives the second claim.
\end{proof}

\subsection{Heuristic: Analogy to Bilinear Hilibert Transform}
The reader at this point may think the symbol of the operator we have just developed is smooth and should not be analyzed as follows --- however, we stress that carving it in a naive way will be problematic to analyze because there are two scale parameters which interact.  So we take what seems, at first glance, to be a rather obtuse approach.  Ignoring the absolute value signs, we may take the Fourier transform and inverse Fourier transform to produce the following Fourier representation of our operator:
\[
\sup_{k_1, k_2}\left|\int_{\mathbb{R}^3}\hat{f_1}(\xi_1)\hat{f_2}(\xi_2)\hat{f_3}(\xi_3)\alpha(2^{k_1}(\xi_1-\xi_2))\beta(2^{k_2}(\xi_3-\xi_2))e^{2\pi ix(\xi_1+\xi_2+\xi_3)}d\xi_1d\xi_2d\xi_3\right|,
\]
where $\alpha$ and $\beta$ are of the type given in the previous lemma.  It will be more convenient later to reverse the sign of the argument of $\alpha$, which is harmless, and so we change $\alpha(s)$ to $\alpha(-s)$.  Suppose for the moment that $\alpha$ were constant in a small neighborhood of the origin --- this is actually impossible since
\[
\Delta\alpha(0)=\int\widehat{\Delta\alpha}(\xi)e^{2\pi i\cdot 0\cdot \xi}d\xi=-(2\pi)^2\int\xi^2\hat{\alpha}(\xi)d\xi<0,
\]
by the positivity of $\hat{\alpha}$.  Ignoring this technical difficulty, we would have that $\alpha(0)-\alpha(s)$ is a function equal to $\alpha(0)$ for $|s|\ge 1$ and $0$ in a neighborhood of the origin.  The bilinear symbol $\alpha(0)-\alpha(\xi_1-\xi_2)$ restricted to $\xi_1<\xi_2$ then looks something like a constant multiple of a scale-truncated Bilinear Hilbert transform --- the Bilinear Hilbert transform's symbol is something like $\chi_{\xi_1<\xi_2}$; if one broke this function up scale by scale according to a Littlewood--Paley partition of unity (with respect to the line $\xi_1=\xi_2$), the $\alpha$ we are now encountering is analogous to a sum over all the scales above $1$.  Of course we actually have two symbols, $\alpha(\xi_1-\xi_2)$ and $\beta(\xi_3-\xi_2)$, which interact with one another.  Since the parameters $k_1$ and $k_2$ are independent scale parameters, this gives the impression that our operator corresponds to something like a doubly maximal-variant of two interacting Bilinear Hilbert transforms.  Ignoring the maximal nature of such an object, the Biest operator studied by Muscalu, Tao, and Thiele in \cite{MTTBiest1} is of a similar type.  Thus there is some hope of borrowing some of their techniques to deal with the present issues.

\subsection{Making the Analogy Precise}

As indicated above, we would prefer if, say, the function $\alpha$ produced in the previous lemma were actually constant in a neighborhood of zero.  This is not directly possible.  However, we may produce an acceptable substitute via the following technical lemma, which is a slightly modified version of \cite[Theorem 3.1]{DTT}:

\begin{lemma}\label{compactinfrequency}
Suppose that $\tilde{\alpha}$ and $\tilde{\beta}$ are both constant in $[-1,1]$ and zero outside $[-2,2]$, and
\begin{align*}
\left|\left(\tilde{\alpha}\right)^{\vee}(s)\right|\lesssim\frac{1}{\left(1+|s|\right)^{M_1}}\\
\left|\left(\tilde{\beta}\right)^{\vee}(t)\right|\lesssim\frac{1}{\left(1+|t|\right)^{M_2}}.
\end{align*}
If we can show that $T^{*}_{\tilde{\alpha},\tilde{\beta}}$ satisfies the desired estimates, depending on $M_1,M_2$ and on the implied constants in the two inequalities above but not on the particular $\alpha,\beta$, then we may pass these estimates to the operators $T^{*}_{\alpha,\beta}$ above.
\end{lemma}

\begin{proof}
Let $\tilde{\alpha}$ be a smooth, symmetric function which is identically $1$ on [-1,1] and supported on $[-2,2]$.  Write 
\[
\alpha(\xi)=\tilde{\alpha}(\xi)+\sum_{u=-\infty}^{0}\phi_u(\xi)
\]
where 
\[
\phi_u(\xi)=\left(\alpha(\xi)-\tilde{\alpha}(\xi)\right)\left(\tilde{\alpha}(\xi/2^u)-\tilde{\alpha}(\xi/2^{u-1})\right).
\]
Perform a similar construction for $\beta$ using $\tilde{\beta}$ and $\varphi_v$.  Then by the triangle inequality, we have the following pointwise estimate:
\begin{equation}
T^*_{\alpha,\beta}\lesssim T^*_{\tilde{\alpha},\tilde{\beta}}+\sum_{v=-\infty}^0T^*_{\tilde{\alpha},\varphi_v}+\sum_{u=-\infty}^0T^*_{\phi_u,\tilde{\beta}}+\sum_{u,v=-\infty}^0T^*_{\phi_u,\varphi_v}.\label{varphidummy1}
\end{equation}
The first term on the right of (\ref{varphidummy1}) obviously satisfies the conditions of the lemma.  We now focus on the second term.  Observe that $\varphi_v$ is identically zero on $[-2^v,2^v]$ and also when $|\xi|\ge 2\times 2^v$; a similar statement holds for $\varphi_v$.  Now, since
\[
2^v\widehat{\varphi_v(2^v\cdot)}(\xi)=\widehat{\varphi_v}(2^{-v}\xi),
\]
it follows that 
\begin{equation*}
T^{*}_{\tilde{\alpha},\varphi_v(2^v\cdot)}=T^{*}_{\tilde{\alpha},\varphi_v},
\end{equation*}
and so
\[
T^{*}_{\tilde{\alpha},\varphi_v}=2^vT^{*}_{\tilde{\alpha},\varphi_v(2^v\cdot)/2^v}.
\]
Now, we know that $\phi_v(2^v\cdot)/2^v$ is supported inside $[-2,2]$ and is constant on $[-1,1]$.  Moreover, we have that
\begin{equation}
\left({\varphi_v(2^v\cdot)}\right)^{\vee}=2^{-v}\check{\varphi}_v\left(\frac{\xi}{2^v}\right).\label{varphidummy2}
\end{equation}
By writing $\check{\varphi}_v$ as a convolution and putting the modulus inside the integral from the convolution, it is easy to see that
\[
\frac{1}{2^v}\left|\check{\varphi}_v(s)\right|\lesssim 2^v\frac{1}{\left(1+|s|\right)^{M_1}}\|\tilde{\alpha}\|_2,
\]
where the implied constant depends on $\alpha$ but not $v$.  Plugging this into (\ref{varphidummy2}), we have that
\[
\frac{1}{2^v}\left|\left({\varphi_v(2^v\cdot)}\right)^{\vee}(\xi)\right|\lesssim\frac{1}{\left(1+|s|\right)^{M_1}}\|\tilde{\alpha}\|_2.
\]
This, together with the definition of $\tilde{\alpha}$, guarantees that $T^{*}_{\tilde{\alpha},\varphi_v(2^v\cdot)/2^v}$ satisfies all the conditions in the statement of the lemma; hence we can translate estimates on $T^{*}_{\tilde{\alpha},\varphi_v(2^v\cdot)/2^v}$ to
\[
\sum_{v=-\infty}^02^v\left(T^{*}_{\tilde{\alpha},\varphi_v(2^v\cdot)/2^v}\right)=\sum_{v=-\infty}^0T^*_{\tilde{\alpha},\varphi_v},
\] 
which takes care of the second term on the right of (\ref{varphidummy1}).  The last two terms are dealt with in a similar manner.
\end{proof}

The above lemma allows us to assume that the functions $\alpha$ and $\beta$ appearing in our operator are supported in $[-2,2]$ and constant in $[-1,1]$.  In fact, the lemma allows us to assume that they are actually either $0$ or $1$ in $[-1,1]$.  Since we may clearly write such a function which is $0$ in $[-1,1]$ as a difference of two functions which are $1$ in $[-1,1]$ and $0$ outside $[-2,2]$, we make the following assumption:

\begin{assumption}
Assume without loss of generality that $\alpha,\beta\equiv 1$ in $[-1,1]$.
\end{assumption}

\subsection{Discretization.}
We recall our object of study:
\[
\sup_{k_1, k_2}\left|\int_{\mathbb{R}^3}\hat{f_1}(\xi_1)\hat{f_2}(\xi_2)\hat{f_3}(\xi_3)\alpha(2^{k_1}(\xi_2-\xi_1))\beta(2^{k_2}(\xi_3-\xi_2))e^{2\pi ix(\xi_1+\xi_2+\xi_3)}d\xi_1d\xi_2d\xi_3\right|,
\]
where $\alpha,\beta$ satisfy the conditions in Lemma (\ref{compactinfrequency}).  We take a different approach to that taken in the maximal multilinear paper by Demeter, Tao, and Thiele, \cite{DTT}.  We first use the triangle inequality to consider separately the integrals over each of the four regions of $\mathbb{R}^3$ determined by the two planes $\xi_1=\xi_2$ and $\xi_2=\xi_3$ --- all four regions are treated identically, so we consider only $\xi_1<\xi_2<\xi_3$.

\subsection{Decomposition Heuristic}
For a fixed $k_1,k_2$, our symbol is $\alpha(2^{k_1}(\xi_2-\xi_1))\beta(2^{k_2}(\xi_3-\xi_2))\chi_{\xi_1<\xi_2<\xi_3}$.  We will make the following imprecise (and incorrect) observations to get a feeling of what kind of model we should expect.  First, in the usual way, one can write $\alpha(2^{k_2}(\xi_2-\xi_1))$ as a cascading sum of functions $\theta_i(\xi_2-\xi_1)$ which are supported on bands where $\xi_2-\xi_1\approx 2^{-i}$, and likewise for $\beta$ and functions $\theta_j'$ supported on the bands $\xi_3-\xi_2\approx 2^{-j}$.  We now split the operator into three pieces, namely where $i\gg j$, $i\approx j$, and $j\gg i$, respectively.  The piece where $i\approx j$ has only one true scale parameter, and thus the techniques of \cite{DTT} are, roughly speaking, sufficient.  By symmetry, it suffices to consider only $i\gg j$.

For each scale $i$, one can, heuristically speaking, write $\theta_i(\xi_2-\xi_1)=\sum_{\ell_1}\phi_i^{\ell_1}(\xi_1)\phi_i^{\ell_1+2}(\xi_2)$, where $\phi_i^s$ is a function supported in an interval $\omega_{i,\ell_1}:=[2^{-i}\ell_1,2^{-i}(\ell_1+1)]$ and are something like the characteristic function of $\omega_{i,\ell_1}$.  This is technically an oversimplification (one truly requires a finite number of expressions involving $\phi_i^{\ell_1}(\xi_1)\phi_i^{\ell_1+n}$, for instance), but we are merely making a heuristic approach anyway, so we ignore these details for the moment.  See Figure \ref{fig:model_heuristic_1} below.
\begin{figure}[!htb]
\centering
\includegraphics[scale=1]{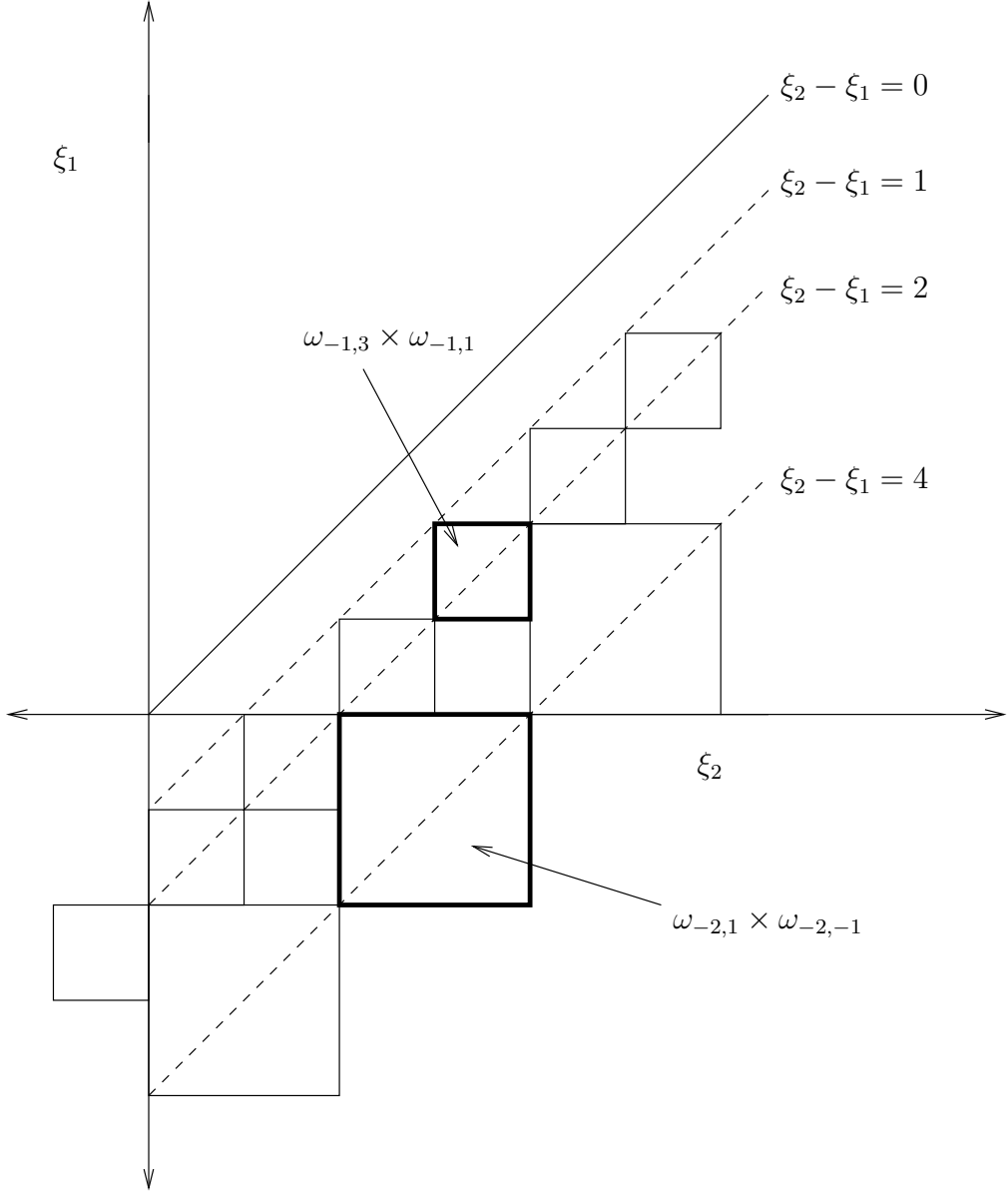}
\caption{A rough visual of how to carve the symbol for $k_2\gg k_1$.}
\label{fig:model_heuristic_1}
\end{figure}

In a similar way, produce functions $\phi_j^{\ell_2}$ for $\theta_j'$.  Then one can break our symbol up as
\begin{align*}
\alpha(2^{k_1}(\xi_1-\xi_2))\beta(2^{k_2}(\xi_3-\xi_2))&=\sum_{i\gg j,i\ge k_1,j\ge k_2}\phi_i^{\ell_1}(\xi_1)\phi_i^{\ell_1+2}(\xi_2)\phi_j^{\ell_2}(\xi_2)\phi_j^{\ell_2+2}(\xi_3)\\
&+\sum_{i\approx j,i\ge k_1,j\ge k_2}\phi_i^{\ell_1}(\xi_1)\phi_i^{\ell_1+2}(\xi_2)\phi_j^{\ell_2}(\xi_2)\phi_j^{\ell_2+2}(\xi_3)\\
&+\sum_{i\ll j,i\ge k_1,j\ge k_2}\phi_i^{\ell_1}(\xi_1)\phi_i^{\ell_1+2}(\xi_2)\phi_j^{\ell_2}(\xi_2)\phi_j^{\ell_2+2}(\xi_3).
\end{align*}
Recall that we only consider the $i\gg j$ region which corresponds to the first term in the above sum.  Now, we have that the supports of $\phi_i^{\ell_1+2}(\xi_2)$ and $\phi_j^{\ell_2}(\xi_2)$ must intersect to produce nonzero terms in this sum, and therefore dyadicity of these intervals and the fact that $i\gg j$ guarantees that $\omega_{i,\ell_1+2}\subset \omega_{j,\ell_2}$.  We now make another technical oversimplification and presume the following completely false equality: $\phi_i^{\ell_1+2}(\xi_2)\phi_j^{\ell_2}(\xi_2)=\phi_i^{\ell_1+2}(\xi_2)$.  This ``makes sense'' since these functions are to be thought of as characteristic functions and, in any case, the $j$-function is roughly constant on the interval for the $i$ function by the separation of scales.  Then our operator looks like
\[
\sum_{i\ge k_1,i\gg j\ge k_2,\omega_{i,\ell_1+2}\subset \omega_{j,\ell_2}}\int_U\widehat{f_1}(\xi_1)\phi_i^{\ell_1}(\xi_1)\widehat{f_2}(\xi_2)\phi_i^{\ell_1+2}(\xi_2)\widehat{f_3}(\xi_3)\phi_j^{\ell_2+2}(\xi_3)e^{2\pi ix(\xi_1+\xi_2+\xi_3)}d\xi,
\]
where the integral is over the region 
\[
U:=\{(\xi_1,\xi_2,\xi_3):\xi_1<\xi_2<\xi_3\}.
\]
We may now ``re-insert'' the supremums into our operator and linearize the problem by considering two arbitrary (but fixed) integer-valued functions $N_1(x)$ and $N_2(x)$ to obtain
\begin{align*}
\sum_{i\gg j,\omega_{i,\ell_1+2}\subset \omega_{j,\ell_2}}\int_U \widehat{f_1}(\xi_1)\phi_i^{\ell_1}(\xi_1)&\widehat{f_2}(\xi_2)\phi_i^{\ell_1+2}(\xi_2)\widehat{f_3}(\xi_3)\phi_j^{\ell_2+2}(\xi_3)\\
& \times e^{2\pi ix(\xi_1+\xi_2+\xi_3)}d\xi1_{i\ge N_1(x)}1_{j\ge N_2(x)}.
\end{align*}
The only caveat is that the estimates must of course be independent of $N_1$ and $N_2$.  If one now dualizes with a function $f_4$ and discretizes in the usual way, i.e. as in \cite{camilbook}, grouping like scales together, one obtains a model of the form
\begin{equation}\label{simplemodel}
\sum_{i\gg j,m_1\in\mathbb{Z}} |\omega_{i,\ell_1+2}|\langle f_1,\check{\phi}_i^{\ell_1,m_1}\rangle \langle f_2,\check{\phi}_i^{\ell_1+2,m_1}\rangle \left\langle M_{i,\ell_1}(f_3)\check{\phi}_i^{2\ell_1+2,m_1}\chi_{\{|\omega_{i,\ell_1+1}|^{-1}\ge 2^{N_1(x)}\}},f_4\right\rangle,
\end{equation}
where
\[
M_{i,\ell_1}(f_3):=\sum_{m_2\in\mathbb{Z},\omega_{i,\ell_1+2}\subset \omega_{j,\ell_2}}\langle f_3,\check{\phi}_j^{\ell_2,m_2}\rangle\check{\phi}_j^{\ell_2,m_2}\chi_{\{|\omega_{j,\ell_2}|^{-1}\ge 2^{N_2(x)}\}}.
\]
Here of course the functions $\check{\phi}_s^{m,n}$ are $L^2$-normalized functions whose Fourier transforms are supported on intervals of length $2^{-s}$ translated $\ell_2\cdot 2^{-s}$ units; moreover, the function itself is ``morally'' localized to an interval of length $2^s$ and translated by $m_2\cdot 2^s$ units. 

Ignoring the factor of $M_{i,\ell_1}(f_3)$ --- i.e. erasing it completely --- one encounters exactly a model of the type found in \cite{DTT}, and their techniques apply directly.  The factor $M_{i,\ell_1}$, for a fixed $m_2$, is something like a localized maximal Hilbert transform which depends on the pair $i,\ell_1$.  One expects, for $m_2$ very different from the corresponding $m_1$, quite a bit of decay so that really only the $m_2\approx m_1$ terms contribute significantly.

The main novelty of the model here is that it has a genuinely bi-parameter structure along with two characteristic functions controlling the scales independently.  Thus the techniques of \cite{DTT} do not apply, and one must obtain new size and energy estimates, which is no small task.

Under the assumption that $i\gg j$ we may invoke the triangle inequality yet again to focus on two separate cases for the supremum: the supremum over $k_1,k_2$ when $k_1>k_2$ and when $k_2\ge k_1$.  In the latter case since we have $i\gg j$, we know that $i\gg j>k_2\ge k_1$, i.e. when $j>k_2$, we automatically have $i>k_1$; thus the supremum can be relaxed to simply a supremum over only $k_2$ in this case.  In the following section, we build the model under the assumption that $k_2\ge k_1$.  The other case is more delicate and will be written up separately.

\subsection{Making the above heuristic precise: a Taylor series approach for $i\gg j$ and $k_2\ge k_1$}

Since $\alpha,\beta$ are constant in [-1,1] and supported in [-2,2], we see that 
\begin{align*}
\theta(s)&=\alpha(s)-\alpha(2s),\\
\phi(t)&=\beta(t)-\beta(2t),\\
\end{align*}
are zero in $[-1,1]$ and outside of $[-2,2]$.  We now write
\begin{align*}
\theta_i(s)&=\theta(2^is),\\
\phi_j(t)&=\phi(2^jt).
\end{align*}
Thus we may write
\begin{align*}
\alpha(2^{k_1}s)=\sum_{i\ge k_1}\theta_i(s),\\
\beta(2^{k_2}t)=\sum_{j\ge k_2}\phi_j(t).
\end{align*}
Hence for any given $f_1,f_2,f_3$, we may write our maximal operator as
\[
\sup_{k_1,k_2}\left|\sum_{i\ge k_1,j\ge k_2}\int_{\mathbb{R}^3}\hat{f_1}(\xi_1)\hat{f_2}(\xi_2)\hat{f_3}(\xi_3)\theta_i(\xi_2-\xi_1)\phi_j(\xi_3-\xi_2)e^{2\pi ix(\xi_1+\xi_2+\xi_3)}d\vec{\xi}\right|,
\]
where $\theta_i(\xi_2-\xi_1)$ and $\phi_j(\xi_3-\xi_2)$ are supported in the bands $|\xi_1-\xi_2|\approx 2^{-i}$ and $|\xi_3-\xi_2|\approx 2^{-j}$, respectively.  As stated previously, we split the interior sum into $i\gg j$, $i\approx j$ and $j\gg i$ and the supremum into the supremum over $k_2\ge k_1$ and $k_2<k_1$.  More precisely, one may consider the sums where $j>i+10$, $i>j+10$ and $|i-j|\le 10$.  Under either assumption that $k_2\ge k_1$ or $k_2<k_1$, the restriction to scales where $i\approx j$ is really a finite sum of single-parameter maximal operators nearly identical to those from the work of Demeter, Tao, and Thiele --- these operators, after a trivial modification, can all be treated using identical techniques to that of \cite{DTT}.  Thus one only needs to consider the four remaining options, which really consist of two pairs of analogous conditions.  Thus it suffices to consider only $i\gg j$ under either the condition $k_2\ge k_1$ or $k_1>k_2$.
\begin{assumption}
For the remainder of our discussion, we consider only the case $i\gg j$, i.e.
\[
\sup_{k_1,k_2}\left|\sum_{i\gg j,i\ge k_1, j\ge k_2}\int_{\mathbb{R}^3}\hat{f_1}(\xi_1)\hat{f_2}(\xi_2)\hat{f_3}(\xi_3)\theta_i(\xi_2-\xi_1)\phi_j(\xi_3-\xi_2)e^{2\pi ix(\xi_1+\xi_2+\xi_3)}d\vec{\xi}\right|,
\]
where $i\gg j$ means $i>j+10$.
\end{assumption}

Moreover, as stated in the title of this section, we will focus only on the case when $k_2\ge k_1$:

\begin{assumption}
For the remainder of our discussion, we discuss only the case $k_2\ge k_1$ and $i\gg j$, i.e.
\[
\sup_{k_2}\left|\sum_{i\gg j\ge k_2}\int_{\mathbb{R}^3}\hat{f_1}(\xi_1)\hat{f_2}(\xi_2)\hat{f_3}(\xi_3)\theta_i(\xi_2-\xi_1)\phi_j(\xi_3-\xi_2)e^{2\pi ix(\xi_1+\xi_2+\xi_3)}d\vec{\xi}\right|,
\]
where $i\gg j$ means that $i> j+10$.
\end{assumption}

It will again be convenient to consider the integral only over the set $U\subset \mathbb{R}^3$ where $\xi_1<\xi_2<\xi_3$ (the other three analogous regions are treated in the same way, modulo a transposition of indices).  In the subset of $U$ where $i\gg j$, we see that any product $\theta_i\phi_j$ is only nonzero in the region $\xi_3-\xi_2\gg\xi_2-\xi_1$ since $\xi_3-\xi_2\approx 2^{-j}\gg 2^{-i}\approx\xi_2-\xi_1$.  One of the basic observations from the Biest paper, \cite{MTTBiest2}, is that in this region, $\chi_{\xi_1<\xi_2<\xi_3}=\chi_{\xi_1<\xi_2}\cdot\chi_{\xi_1+\xi_2<2\xi_3}$.  This latter form is somewhat more convenient: when one discretizes each factor on the right side of this equation, one gets something like $\psi_i^1(\xi_1)\psi_i^2(\xi_2)\psi_j^1(\xi_1+\xi_2)\psi_j^2(\xi_3)$.  This is nicer in the sense that the inverse Fourier transform of this is then
\[
\left((\check\psi_i^1\check\psi_i^2)*\check\psi_j^1\right)\cdot\check\psi_j^2,
\]
which is something like a composition of two bilinear Hilbert transforms, where the ``inner'' BHT is localized to the (larger) frequency interval of the ``outer'' BHT.

In the Biest paper, \cite{MTTBiest2}, Muscalu, Tao, and Thiele are able to subtract from the symbol $\chi_{\xi_1<\xi_2<\xi_3}$ a smooth function which equals $\chi_{\xi_1<\xi_2}\cdot\chi_{\xi_1+\xi_2<2\xi_3}$ in the range $|\xi_3-\xi_2|\gg |\xi_2-\xi_1|$ (as well as a second function performing a similar role where $2\xi_1<\xi_2+\xi_3$ and $\xi_2<\xi_3$) to produce something which is a smooth ``standard symbol'' in that it has only a ``nice'' singularity along the line $\xi_1=\xi_2=\xi_3$ (rather than the two planes $\xi_1=\xi_2$ and $\xi_2=\xi_3$.  We would like to perform a similar dissection of our operator, but our symbol is complicated by the fact that we have something like the symbol for $\chi_{\xi_1<\xi_2<\xi_3}$ which is smoothly \emph{truncated}.  When making a similar approach of subtracting ``nice'' symbols, the fact that this symbol is not identically equal to 1 or 0 has the effect of creating ``boundary'' terms which are quite complicated, requiring different methods which are apparently as difficult as the ones we presently encounter.  We thus veer from the Biest approach somewhat in favor of the following methodology.  We will still encounter error terms, but they will have a more reasonable shape.

By Taylor's theorem, for a smooth function $f$,
\[
f(x)=f(a)+(x-a)f'(a)+...+\frac{(x-a)^n}{n!}f^{(n)}(a)+f_{n}(x-a),
\]
where $f_{n}$ is the remainder from Taylor's theorem.  Thus we may write
\[
\phi_j(\xi_3-\xi_2)=\sum_{m=0}^n\frac{\left(\xi_1-\xi_2\right)^m}{2^mm!}\phi_j^{(m)}\left(\xi_3-\frac{\xi_1+\xi_2}{2}\right)+\psi_{j,n}\left(\xi_3-\frac{\xi_1+\xi_2}{2}\right),
\]
where $\psi_{j,n}$ is the remainder term from Taylor's theorem.  In particular, by the definition of $\phi_j$, it follows that
\[
\phi_j^{(m)}\left(\xi_3-\frac{\xi_1+\xi_2}{2}\right)=2^{jm}\phi^*_{m,j}\left(\xi_3-\frac{\xi_1+\xi_2}{2}\right),
\]
where $\phi^*_{m,j}$ is also a smooth, bounded function supported on the same interval as $\phi_j$.  Moreover, $(\xi_1-\xi_2)^m\approx 2^{-im}$ on the support of $\theta_i$, and so $\theta^*_{i,m}(\xi_2-\xi_1)=2^{im}\theta_i(\xi_2-\xi_1)(\xi_1-\xi_2)^m$ is a smooth, bounded function supported on the same interval as $\theta_i$.  Thus for a fixed pair $i,j$, the $m$-th order term in the Taylor expansion gains a factor of $2^{-m(i-j)}$, which is small when $i-j$ is big --- this holds since we are in the situation that $i\gg j$.  We denote by $\tau_{m,k}(\xi_1,\xi_2,\xi_3)$ the symbol which corresponds to the sum of all products $\theta^*_{i,m}\phi^*_{j,m}$ such that $i-j=k\gg 0$ and $j\ge k_2$.  Since we are assuming that $i>j+10$, we have that $k>10$.  So, the operator whose symbol is the sum of all the $m$-th order terms is given by $\sum_{k>10}2^{-mk}\tau_{m,k}$.  It is not hard to observe that for a finite family of multi-indices $\alpha$, we may pick $m$ large so that
\[
|\partial^\alpha\tau_{m,k}(\xi)|\lesssim 2^{k(m-|\alpha|)}\frac{1}{|\xi|^\alpha},
\]
for all $\alpha$ in this family.  By doing similar computations for the remainder $\psi_{j,n}$ (and using the remainder theorem for Taylor series), one gets a similar result for the symbol $\tau_{n,k}$ (coming from $\psi_{j,n}$).  Thus for sufficiently large $n$, the $\tau_{n,k}$ satisfy the usual condition for the multilinear Coifman--Meyer multiplier theorem (a recent proof may be found in \cite{biparam}).  We cannot apply the theorem directly, however, since we additionally have a supremum over $k_2$ still waiting for us.  However, this is not a major issue.  We will briefly discuss  why this is in the following paragraph.

As one can see using the techniques we will use shortly for the $m=0$ term, the discrete model for $\tau_{n,k}$ will be something like
\[
\sum_P\langle B_{P,k}(f_1,f_2),\phi_P^1\rangle\langle f_3,\phi_P^2\rangle \langle f_4,\phi_P^3 1_{|I_P|>2^{N_2(x)}}\rangle,
\]
where
\[
B_{P,k}(f_1,f_2)=\sum_{Q:\omega_{Q_3}\subset \omega_{P_1},\frac{|I_P|}{|I_Q|}=2^k}\langle f_1,\phi_Q^1\rangle\langle f_2,\phi_Q^2\rangle\phi_Q^3.
\]
Each interval $\omega_{P_1}$ has length $2^{-j}$ and each $\omega_{Q_3}$ has length $2^{-i}$.  Thus there are precisely $2^k$ intervals $\omega_{Q_3}$ that will contribute to the sum.  One can then consider a sum of $2^k$ models, where the $\omega_{Q_3}$ lie in a fixed position within the $\omega_{P_1}$ intervals; if one can estimate each one of these terms separately (in a uniform way), one can estimate the whole model for $\tau_{n,k}$, losing a factor of $2^k$ in the estimates.  As we will discuss, there are sizes and energies available for the $\langle f_3,\phi_P^2\rangle$ term (which is standard) as well as the $\langle f_4,\phi_P^3 1_{|I_P|>2^{N_2(x)}}\rangle$ term (which follows from the methods in \cite{DTT}).  The remaining term, $\langle B_{P,k}(f_1,f_2),\phi_P^1\rangle$ requires a bit more work to estimate fully.  However, one can perform some manipulations, provided $m$ is sufficiently large, using some ideas from \cite{camilflag} and \cite{MTTBiest2}.

The loss of $2^k$ is more problematic when $m=1$ (since we lose a factor of $2^k$ but only gain a factor of $2^{-k}$), but for larger $m$ one will be able to sum over $k$ to get that the full remainder operator, $\sum_{k>10}2^{-mk}\tau_{n,k}$, is indeed bounded.  Thus it truly suffices to consider the ``main term'', when $m=0$, as well as a few small, positive values for $m$.

The Taylor series terms for positive $m$, are, in theory, nicer objects since their symbols have increased in smoothness.  Nevertheless, there are some technical issues, and estimating them seems, at present, to require more robust technology than is currently available; thus they will need to be written elsewhere.  Recent work by J. Jung, \cite{Jung:2013fj}, seems like a fruitful source of inspiration in this direction.  In any case, we shall focus only on the $m=0$ case in the remainder of our discussion.

\begin{assumption}
For the remainder of our discussion, we focus on the operator given by the $m=0$ term in the Taylor expansion described above, i.e. our operator is
\[
\sup_{k_2}\left|\sum_{i\gg j\ge k_2}\int_{U}\hat{f_1}(\xi_1)\hat{f_2}(\xi_2)\hat{f_3}(\xi_3)\theta_i(\xi_2-\xi_1)\phi_j\left(\xi_3-\frac{\xi_1+\xi_2}{2}\right)e^{2\pi ix(\xi_1+\xi_2+\xi_3)}d\vec{\xi}\right|,
\]
where $i\gg j$ means $i-j>10$ and $U$ is the subspace of $\mathbb{R}^3$ where $\xi_1<\xi_2<\xi_3$.
\end{assumption}

If we dualize with a function $f_4$, we observe that this last line may be majorized by
\begin{align*}
\left|\sum_{i\gg j}\int\right.&\left.\int_{U}\hat{f_1}(\xi_1)\hat{f_2}(\xi_2)\hat{f_3}(\xi_3)\theta_i(\xi_2-\xi_1)\right.\\ &\left.\phi_j\left(\xi_3-\frac{\xi_1+\xi_2}{2}\right)e^{2\pi ix(\xi_1+\xi_2+\xi_3)}f(x)1_{j\ge N_2(x)}d\xi dx\right|,
\end{align*}
for some integer-valued function $N_2(x)$.  Thus it suffices to establish estimates for the above which are independent of $N_2(x)$, which we now fix.

\begin{assumption}
It suffices to estimate
\begin{align*}
\left|\sum_{i\gg j}\int\right.&\left.\int_{U}\hat{f_1}(\xi_1)\hat{f_2}(\xi_2)\hat{f_3}(\xi_3)\theta_i(\xi_2-\xi_1)\right.\\ &\left.\phi_j\left(\xi_3-\frac{\xi_1+\xi_2}{2}\right)e^{2\pi ix(\xi_1+\xi_2+\xi_3)}f(x)1_{j\ge N_2(x)}d\xi dx\right|,
\end{align*}
independent of $N_2(x)$, which is an integer-valued function.
\end{assumption}

To continue further, we will need to make several standard definitions; we group them together in the following section.

\subsection{Notation and Definitions}
We make the following definitions, which are due to Muscalu, Tao, and Thiele; these statements are copied more or less verbatim from \cite[Definitions 4.1--4.6]{MTTBiest2}.
\begin{definition}
Let $n\ge 1$ and $\sigma\in\{0,1/3,2/3\}^n$.  We define the shifted $n$-dyadic mesh $D=D_\sigma^n$ to be the collection of cubes of the form
\[
D_\sigma^n:=\{2^j(k+(0,1)^n+(-1)^j\sigma):j\in\mathbb{Z}\textrm{ and }k\in\mathbb{Z}^n\}.
\]
We define a shifted dyadic cube to be any member of a shifted n-dyadic mesh.
\end{definition}
In the context of our discussion, we will primarily deal with the $n=3$ case.  One can make the standard observation that for any cube $Q$ there exists a shifted dyadic cube $Q'$ such that $Q\subseteq \frac{7}{10}Q'$ and $|Q'|\sim |Q|$.
\begin{definition}
A subset $D'$ of a shifted $n$-dyadic grid $D$ is called sparse if, for any two cubes $Q,Q'$ in $D$ with $Q\neq Q'$, we have $|Q|<|Q'|$ implies $|10^9Q|<|Q'|$ and $|Q|=|Q'|$ implies $10^9Q\cap 10^9Q'=\emptyset$.
\end{definition}
A standard observation is that any subset of a shifted $n$-dyadic grid can be split into $O(1)$ sparse subsets.
\begin{definition}
Let $\sigma=(\sigma_1,\sigma_2,\sigma_3)\in\{0,1/3,2/3\}^3$, and let $1\le i\le 3$.  An $i$-tile with shift $\sigma_i$ is a rectangle $I_P\times\omega_P$ with area 1 and with $I_P\in D_0^1$ and $\omega_P\in D_{\sigma_i}^1$.  A tri-tile with shift $\sigma$ is then a 3-tuple $\vec{P}=(P_1,P_2,P_3)$ such that each $P_i$ is an $i$-tile with shift $\sigma_i$ and the $I_{P_i}=I_{\vec{P}}$ are independent of $i$.  The frequency cube $Q_{\vec{P}}$ is defined to be $\prod_{i=1}^3\omega_{P_i}$.
\end{definition}
We shall sometimes abuse notation and refer to $i$-tiles with shift $\sigma$ as simply $i$-tiles or just tiles if it is unimportant or clear from context what the parameters $\sigma$ and $i$ are.
\begin{definition}
A set $\vec{\mathbf{P}}$ of tri-tiles is called sparse if all tri-tiles in $\vec{P}$ have the same shift and the set $\{Q_{\vec{P}}:\vec{P}\in\vec{\mathbf{P}}\}$ is sparse.
\end{definition}
Clearly by the previous observation, any set of tri-tiles can be split into $O(1)$ sparse subsets.
\begin{definition}
Let $P$ and $P'$ be tiles.  We write $P'<P$ if $I_{P'}\subsetneq I_P$ and $3\omega_P\subseteq 3\omega_{P'}$, and $P'\le P$ if $P'<P$ or $P'=P$.  We write $P'\lesssim P$ if $I_{P'}\subseteq I_P$ and $10^7\omega_P\subseteq 10^7\omega_{P'}$.  We write $P'\lesssim'P$ if $P'\lesssim P$ and $P'\not\le P$.
\end{definition}
The ordering $<$ is in the spirit of that in Fefferman, \cite{Feff2}, or Lacey and Thiele, \cite{LTBHT2}, \cite{LTCC}, \cite{Thiele}, but slightly different as $P'$ and $P$ do not quite have to intersect.  This is more convenient for technical purposes.
\begin{definition}\label{rank1def}
A collection $\vec{\mathbf{P}}$ of tri-tiles is said to have rank 1 if one has the following properties for all $\vec{P},\vec{P}'\in\vec{\mathbf{P}}$:
\begin{enumerate}
\item If $\vec{P}\neq\vec{P}'$, then $P_j\neq P_j'$ for all $j=1,2,3$.
\item If $P'_j\le P_j$ for some $j=1,2,3$, then $P_i'\lesssim P_i$ for all $1\le i\le 3$.
\item If in addition to $P_j'\le P_j$ for some $j$ we assume that $|I_{\vec{P}'}|<10^9|I_{\vec{P}}|$, then we have $P_i'\lesssim' P_i$ for all $i\neq j$.
\end{enumerate}
\end{definition}
\begin{definition}
Let $P$ be a tile.  A wave packet adapted to $P$ is a function $\phi_P$ which has Fourier support in $\frac{9}{10}\omega_P$ and obeys the estimates
\[
|\phi_P(x)|\lesssim |I_P|^{-1/2}\tilde{\chi}_{I_P}(x)^M
\]
for all $M>0$, where the implicit constant of course depends on $M$ and where
\[
\tilde{\chi}_I(x):=\left(1+\left(\frac{|x-x_I|}{|I|}\right)^2\right)^{-1/2},
\]
where $x_I$ is the center of the interval $I$.
\end{definition}

\subsection{Building the model for $m=0$ when $i\gg j$ and $k_2\ge k_1$}

To reiterate, we are now considering
\begin{align*}
\left|\sum_{i\gg j}\int\right.&\left.\int_{U}\hat{f_1}(\xi_1)\hat{f_2}(\xi_2)\hat{f_3}(\xi_3)\theta_i(\xi_2-\xi_1)\right.\\ &\left.\phi_j\left(\xi_3-\frac{\xi_1+\xi_2}{2}\right)e^{2\pi ix(\xi_1+\xi_2+\xi_3)}f(x)1_{j\ge N_2(x)}d\xi dx\right|,
\end{align*}
where we $U$ is the subspace of $\mathbb{R}^3$ where $\xi_1<\xi_2<\xi_3$.

We now proceed through some standard computations.  First, we note that $\theta_i(\xi_2-\xi_1)$ is supported on the set where $\xi_2-\xi_1\in [2^{-i},2^{-i+1}]$ (recall that we are only considering $\xi_1<\xi_2<\xi_3$, and so we ignore the fact that $\theta_i$ is actually also nonzero on $[-2^{-i+1},-2^{-i}]$).  We cover this region with a family of shifted dyadic squares, $\textbf{Q}_\sigma$, where each $Q\in\mathbf{Q}_\sigma$ satisfies $d(Q,\{\xi_1=\xi_2\})\approx 2^{-i}$, so that the side length of $Q$, which we denote $|Q|$, is also approximately $2^{-i-10}$.  Now produce a family of functions $\psi_{Q,1}(\xi_1),\psi_{Q,2}(\xi_2)$ so that $\psi_{Q,t}$ is supported on $\frac{8}{10}Q_t$ and $\check\psi_{Q,t}$ are each adapted to a dyadic interval $I_Q$ (with $|I_Q|=1/|Q|$) and have $\|\check\psi_{Q,t}\|_1\lesssim 1$.  For example, one can construct a function $\gamma$ which is non-negative and supported on $[0.2,0.8]$ which decays arbitrarily rapidly away from the origin (since it is necessarily a Schwartz function) and such that
\[
\sum_\ell \left|\gamma\left(\xi-\frac{\ell}{3}\right)\right|^2=1.
\]
This is possible because the intervals $[0.2,0.8]$ translated by multiples of $1/3$ cover the line with enough room for smooth cutoffs.  The translation by $\ell/3$ adds a complex exponential to the inverse Fourier transform, which does not affect adaptedness.  Thus these functions will suffice.  Since we are thinking of these functions as being related to the frequency intervals corresponding to the sides of $Q$, we will denote these by $\omega_{Q_1}$ and $\omega_{Q_2}$, respectively.  By these observations, we can choose the $\psi_{Q,t}$ in such a way that 
\[
a(\xi_1,\xi_2):=\sum_{\sigma\in\{0,1/3,2/3\}^2}\sum_{Q\in\mathbf{Q}_{\sigma,i}}\psi_{Q,1}(\xi_1)\psi_{Q,2}(\xi_2)
\]
satisfies
\[
a(\xi_1,\xi_2)\equiv 1\textrm{, when }\xi_2-\xi_1\in [2^{-i},2^{-i+1}].
\]
Then
\[
\theta_i(\xi_2-\xi_1)=\sum_{\sigma\in\{0,1/3,2/3\}}\sum_{Q\in\mathbf{Q}_{\sigma,i}}\theta_i(\xi_2-\xi_1)\psi_{Q,1}(\xi_1)\psi_{Q,2}(\xi_2)
\]
Let $|Q|:=2^{-i}$.  Also, let $\tilde{\psi}_{Q,t}(\xi_t)$ denote a function whose inverse Fourier transform is $L^1$-normalized and adapted to the same interval $I_Q$ as $\psi_{Q,t}(\xi_t)$ which is $1$ on $\frac{8}{10}\omega_{Q_t}$ and $0$ outside of $\frac{8.5}{10}\omega_{Q_t}$.  Identifying $Q$ with $\mathbb{T}^2$ in the obvious way, we compute a Fourier series to see that
\[
\theta_i(\xi_2-\xi_1)\psi_{Q,1}(\xi_1)\psi_{Q,2}(\xi_2)=\sum_{n_1,n_2}C^Q_1(n_1,n_2)e^{2\pi i{\frac{n_1}{|Q|}}\xi_1}e^{2\pi i\frac{n_2}{|Q|}\xi_2},
\]
on the support of $\tilde{\psi}_{Q,1}(\xi_1)\tilde{\psi}_{Q,2}$.
Hence
\[
\theta_i(\xi_2-\xi_1)=\sum_{n_1,n_2}\sum_{\sigma\in\{0,1/3,2/3\}}\sum_{Q\in\mathbf{Q}_{\sigma,i}}C_1^Q(n_1,n_2)\tilde{\psi}_{Q,1}(\xi_1)\tilde{\psi}_{Q,2}(\xi_2).
\]
\begin{lemma}
$C_1^Q(n_1,n_2)$ depends only on the $\sigma$ in the definition of $\mathbf{Q}_{\sigma,i}$ rather than individual $i$; moreover, it decays arbitrarily rapidly in $n_1,n_2$.  In particular,
\[
|C_1^Q(n_1,n_2)|\lesssim \frac{1}{(1+|n|)^{M+10}},
\]
where $M$ is the decay rate in the definition of a function being adapted to an interval.  Lastly, it can be assumed that $C_1^Q(n_1,n_2)$ does not depend on $Q$, modulo a harmless, finite adjustment of $\mathbf{Q}_{\sigma,i}$ and corresponding finite loss in the estimates.  Thus we replace it with $C_1(n_1,n_2)$.
\end{lemma}
\begin{proof}
We see
\[
C^Q_{n_1,n_2}=\frac{1}{|Q|^2}\int_{\omega_{Q_1}\times\omega_{Q_2}}Q \theta_i(\xi_2-\xi_1)\psi_{Q,1}(\xi_1)\psi_{Q,2}(\xi_2)e^{-\frac{2\pi i}{|Q|}(n_1\xi_1+n_2\xi_2)}d\xi_1d\xi_2.
\]
Apply the change of variable $(\xi_1,\xi_2)\mapsto(|Q|\xi_1,|Q|\xi_2)$, one has
\[
C^Q_{n_1,n_2}=\int_{I_1\times I_2} \theta(\xi_2-\xi_1)\psi_{I_1,1}(\xi_1)\psi_{I_2,2}(\xi_2)e^{-2\pi i(n_1\xi_1+n_2\xi_2)}d\xi_1d\xi_2,
\] 
where $\theta$ lives at scale $1$, and the functions $\psi_{I_t,t}(\xi_t)$ live on intervals $I_1$ and $I_2$ of scale $1$.  Moreover, $\theta$ is independent of $Q$.  The integral then depends on the difference between the relevant $\sigma_i$'s involved as well as the distance between centers of the intervals $I_1$ and $I_2$ --- once one fixes this difference, the integral is always over some rectangle like a fixed $I_1\times I_2$ except translated parallel to $\xi_1=\xi_2$, which does not affect the integral.  But there are only a finite number of possible distances between the centers (by considering the supports relative to $\theta$, and, modulo a finite loss in the estimates, we may assume the distance is fixed).  Repeated applications of integration by parts give the second claim.
\end{proof}

We also write
\[
\phi_{Q,t}(\xi_t):=\frac{1}{1+|n_t|^M}\tilde{\psi}_{Q,t}(\xi_t)e^{2\pi i\frac{n_t}{|Q|}\xi_t},
\]
and observe the following:
\begin{lemma}
$\phi_{Q,t}(\xi_t)$ is a wave packet adapted to $I_Q\times \omega_{Q_t}$ and has $\|\check\phi_{Q,t}\|\lesssim 1$.
\end{lemma}
Thus we finally write
\[
\theta_i(\xi_2-\xi_1)=\sum_{n_1,n_2}\sum_{\sigma}C_1(n_1,n_2)(1+|n_1|^M)(1+|n_2|^M)\sum_{Q\in\mathbf{Q}_{\sigma,i}}\phi_{Q,1}(\xi_1)\phi_{Q,2}(\xi_2).
\]
It is also clear that for a fixed $\xi_1,\xi_2$, only finitely many terms in the sum will be nonzero.  Performing a similar decomposition to the function $\phi_j(a-b)$ and replacing $a=\xi_3$ and $b=\frac{\xi_1+\xi_2}{2}$, one can write
\[
\phi_j\left(\xi_3-\frac{\xi_1+\xi_2}{2}\right)=\sum_{n_3,n_4}\sum_{\sigma'\in\{0,1/3,2/3\}}\tilde{C}_2(n_3,n_4)\sum_{P\in\mathbf{P}_{\sigma',j}}\phi_{P,1}(\xi_3)\phi_{P,2}\left(\frac{\xi_1+\xi_2}{2}\right),
\]
where the $\tilde{C}_2$ has incorporated the polynomial in $n_3,n_4$ which is present in the previous equation.  Hence
\begin{lemma}
Our 4-linear form
\begin{align*}
\left|\sum_{i\gg j}\int\right.&\left.\int_{\mathbb{R}^3}\hat{f_1}(\xi_1)\hat{f_2}(\xi_2)\hat{f_3}(\xi_3)\theta_i(\xi_2-\xi_1)\right.\\ &\left.\phi_j\left(\xi_3-\frac{\xi_1+\xi_2}{2}\right)e^{2\pi ix(\xi_1+\xi_2+\xi_3)}f(x)1_{j\ge N_2(x)}d\xi dx\right|,
\end{align*}
can be written as
\begin{align*}
&\sum_{n\in\mathbb{Z}^4}\sum_{\sigma,\sigma'}C(n)\left|\sum_{i\gg j}\sum_{Q\in\mathbf{Q}_{\sigma,i},P\in\mathbf{P}_{\sigma',j}}\int\int_{\mathbf{R}^3}\hat{f_1}(\xi_1)\hat{f_2}(\xi_2)\hat{f_3}(\xi_3)\right.\\
&\left.\phi_{Q,1}(\xi_1)\phi_{Q,2}(\xi_2)\phi_{P,1}(\xi_3)\phi_{P,2}\left(\xi_1+\xi_2\right)e^{2\pi ix(\xi_1+\xi_2+\xi_3)}f(x)1_{|I_P|\ge 2^{N_2(x)}}d\xi dx\right|,
\end{align*}
and it suffices to consider this or a fixed $n$ and $\sigma,\sigma'$, i.e.
\begin{align*}
&\left|\sum_{i\gg j}\sum_{Q\in\mathbf{Q}_{\sigma,i},P\in\mathbf{P}_{\sigma',j}}\int\int_{\mathbf{R}^3}\hat{f_1}(\xi_1)\hat{f_2}(\xi_2)\hat{f_3}(\xi_3)\right.\\
&\left.\phi_{Q,1}(\xi_1)\phi_{Q,2}(\xi_2)\phi_{P,1}(\xi_3)\phi_{P,2}\left(\frac{\xi_1+\xi_2}{2}\right)e^{2\pi ix(\xi_1+\xi_2+\xi_3)}f(x)1_{|I_P|\ge 2^{N_2(x)}}d\xi dx\right|.
\end{align*}
\end{lemma}

Now, since the inverse Fourier transform of
\[
\phi_{Q,1}(\xi_1)\phi_{Q,2}(\xi_2)\phi_{P,2}\left(\frac{\xi_1+\xi_2}{2}\right)
\]
is
\[
\left(\check\phi_{Q,1}\check\phi_{Q,2}\right)*\check\phi_{P,2},
\]
it follows that we may insert an $L^1$-normalized function $\phi_{Q,3}(\xi_1+\xi_2)$ which is $1$ on the shifted dyadic interval $\frac{8}{10}\omega_{Q_3}:=\frac{8}{10}(\omega_{Q_1}+\omega_{Q_2})$ and $0$ outside $\frac{9}{10}\omega_{Q_3}$.  Since $|\omega_{P_2}|\gg |\omega_{Q_3}|$, we must have that $\omega_{Q_3}\subset \omega_{P_2}+\omega_{P_2}:=\omega_{\tilde{P_2}}$ for the product $\phi_{P,2}\phi_{Q,3}$ to be nonzero.

Carrying the inverse Fourier transform through, we produce
\begin{align*}
&\left|\sum_{i\gg j}\sum_{P\in\mathbf{P}_{\sigma',j}}\sum_{Q\in\mathbf{Q}_{\sigma,i}:\omega_{Q_3}\subset \omega_{P_2}}\int(f_3*\check\phi_{P,1})(x)\right.\\ 
&\left.\left((f_1*\check\phi_{Q,1})(f_2*\check\phi_{Q,2})\right)*\check\phi_{Q,3}*\check\phi_{P,2}(x)f_4(x)1_{|I_P|\ge 2^{N_2(x)}} dx\right|.
\end{align*}
One may also insert a factor $\phi_{P,3}$ which is $1$ on $\frac{8}{10}\omega_{P_3}:=\frac{8}{10}(\omega_{P_1}+\omega_{P_2})$ and $0$ outside $\frac{9}{10}\omega_{P_3}$, to produce
\begin{align*}
&\left|\sum_{i\gg j}\sum_{P\in\mathbf{P}_{\sigma',j}}\sum_{Q\in\mathbf{Q}_{\sigma,i}:{\omega_{Q_3}\subset \omega_{P_2}}}\int(f_3*\check\phi_{P,1})(x)\right.\\ 
&\left.\left((f_1*\check\phi_{Q,1})(f_2*\check\phi_{Q,2})\right)*\check\phi_{Q,3}*\check\phi_{P,2}(x)\left(f_41_{|I_P|\ge 2^{N_2(x)}}\right)*\phi_{P,3}(x) dx\right|.
\end{align*}
Now, perform a standard discretization procedure with respect to $P$, as in \cite[p. 1654--1656]{eyvi}, to produce
\begin{align*}
&\left|\int_0^1\sum_{i\gg j}\sum_{P\in\mathbf{P}_{\sigma',j}}\sum_{Q\in\mathbf{Q}_{\sigma,i}:{\omega_{Q_3}\subset \omega_{P_2}}}\sum_{I_P:|I_P|=|P|^{-1}}\frac{1}{|I_P|^{1/2}}\langle f_3, \phi_{P,1,\alpha}\rangle\right.\\ 
&\left.\left\langle (f_1*\check\phi_{Q,1})(f_2*\check\phi_{Q,2}))*\check\phi_{Q,3}, \phi_{P,2,\alpha}\right\rangle\left\langle f_41_{|I_P|\ge 2^{N_2(x)}},\phi_{P,3,\alpha}\right\rangle d\alpha\right|,
\end{align*}
and perform a second discretization with respect to $Q$:
\begin{align*}
&\left|\int_0^1\int_0^1\sum_{i\gg j}\sum_{P\in\mathbf{P}_{\sigma',j}}\sum_{Q\in\mathbf{Q}_{\sigma,i}:{\omega_{Q_3}\subset \omega_{P_2}}}\sum_{I_P:|I_P|=|P|^{-1}}\frac{1}{|I_P|^{1/2}}\langle f_3, \phi_{P,1,\alpha}\rangle\right.\\ 
&\left.\left\langle \sum_{I_Q:|I_Q|=|Q|^{-1}}\frac{1}{|I_Q|^{1/2}}\langle f_1, \phi_{Q,1,\beta}\rangle\langle f_2, \phi_{Q,2,\beta}\rangle\phi_{Q,3,\beta}, \phi_{P,2,\alpha}\right\rangle\right.\\
&\left.\left\langle f_41_{|I_P|\ge 2^{N_2(x)}},\phi_{P,3,\alpha}\right\rangle d\alpha d\beta\right|.
\end{align*}
Here,
\[
\phi_{P,t,\alpha}(x)=|I_P|^{1/2}\overline{\check\phi_{P,t}(x-\alpha)}
\]
and
\[
\phi_{Q,t,\beta}(x)=|I_Q|^{1/2}\overline{\check\phi_{Q,t}(x-\beta)}
\]
are both $L^2$-normalized bump functions adapted to the tile $I_P\times P_t$ and $I_Q\times Q_t$, respectively, \emph{uniformly} in $\alpha$ and $\beta$.  If we let 
\[
\mathbf{P}:=\{I_P\times P:I_P\textrm{ dyadic},|I_P|=2^{j}, P\in\bigcup_{j}\mathbf{P}_{\sigma',j}\textrm{ for some }j\in\mathbb{Z}\}
\]
and
\[
\mathbf{Q}:=\{I_Q\times Q:I_Q\textrm{ dyadic},|I_Q|=2^{i},Q\in\bigcup_{i}\mathbf{Q}_{\sigma,i}\textrm{ for some }j\in\mathbb{Z}\}
\]
then it suffices to study
\[
\left|\sum_{P\in\mathbf{P}}\frac{1}{|I_P|^{1/2}}\langle f_3, \phi_{P,1,\alpha}\rangle\left\langle B_P(f_1,f_2), \phi_{P,2,\alpha}\right\rangle\left\langle f_41_{|I_P|\ge 2^{N_2(x)}},\phi_{P,3,\alpha}\right\rangle\right|,
\]
where
\[
B_P(f_1,f_2):=\sum_{Q\in\mathbf{Q}:\omega_{Q_3}\subset \omega_{P_2}}\frac{1}{|I_Q|^{1/2}}\langle f_1, \phi_{Q,1,\beta}\rangle\langle f_2, \phi_{Q,2,\beta}\rangle\phi_{Q,3,\beta}.
\]

\begin{definition}
Let $\vec{\mathbf{P}}$ denote the collection of tri-tiles $\vec{P}$ corresponding to the above construction, and likewise for $\vec{\mathbf{Q}}$.
\end{definition}

\begin{proposition}
Modulo a harmless refinement, the families $\vec{\mathbf{P}}$ and $\vec{\mathbf{Q}}$ are sparse and have rank 1 (see Definition \ref{rank1def}).  We may also assume that $\sigma_1=\sigma_2=\sigma_1'=\sigma_2'=0$.
\end{proposition}
\begin{proof}
With a loss of a factor $3^2$, we may assume that $\sigma_1=\sigma_2$.  We also may assume that they are both 0; the other cases are handled precisely the same, modulo some minor changes of notation.  Also, by a refinement and loss of $O(1)$ in the estimates, we may freely assume the two families are sparse.  We prove the rank 1 condition only for $\vec{\mathbf{P}}$, but the proof works identically for $\vec{\mathbf{Q}}$.  We prove each of the three parts of Definition \ref{rank1def} separately.
\begin{enumerate}
\item To establish (1) in the definition, suppose that $P_1=P_1'$, say.  Then clearly the scales of the tiles must be the same; suppose this scale is $j$.  Supposing that the functions $\phi_{P,t}$ live on intervals of slightly smaller scale, say $2^{-j-5}$, then by the construction above, if $\xi_1\in P_1=[2^{-j-5}\ell_1,2^{-j-5}(\ell_1+1)]$ and $\xi_2\in  P_2=[2^{-j-5}\ell_2,2^{-j-5}(\ell_2+1)]$ then from the fact that $\xi_2-\xi_1\in [2^{-j},2^{-j+1}]$ (by the factor of $\phi_j(\xi_2-\xi_1)$), it is easy to deduce that $\ell_2-\ell_1$ can only be selected from a finite family of positive integers (which are nonzero as well).  Thus we may lose a finite factor in the estimates and assume that $\ell_2=\ell_1+n$ for some fixed positive integer $n$, which is away from zero.  Thus given a $P_1$, there is exactly one $P_2$, and hence $P_2=P_2'$.  The definition of $P_3$ is $P_1+P_2$, so we know $P_3=P_3'$ as well.  The other two possible cases follow in a similar fashion.
\item Suppose that for some $t$, $P_t'\le P_t$.  By the previous step, we may assume they are not equal, hence $I_{P'}\subsetneq I_P$ and $3\omega_{P_t}\subset 3\omega_{P_t'}$.  Then certainly, $10^7P_s\lesssim 10^7P_{s'}$.
\item The $P_t$ intervals are separated by a large number of units of length $|I_P|^{-1}$, and so the third part of the definition holds. 
\end{enumerate}
\end{proof}

By the uniformity of adaptedness in $\alpha,\beta$, we may drop the dependence on $\alpha,\beta$ and will write simply $\phi_P^1:=\phi_{P,1,\alpha}$, since the presence of $\alpha$ does not affect the adaptedness of $\phi_{P,1,\alpha}$ to $I_P\times P_1$.  The usual limiting arguments suffice to reduce to finite subsets of $\vec{\mathbf{P}}$ and $\vec{\mathbf{Q}}$.
\begin{assumption}
We are now free to study the following for finite families of rank 1 tiles $\vec{\mathbf{P}}$ and $\vec{\mathbf{Q}}$ and functions $\phi_P^t$ and $\phi_Q^t$ which are $L^2$-normalized and adapted in the appropriate way:
\[
\left|\sum_{P\in\vec{\mathbf{P}}}\frac{1}{|I_P|^{1/2}}\langle f_3, \phi_P^1\rangle\left\langle B_P(f,g), \phi_P^2\right\rangle\left\langle f_41_{|I_P|\ge 2^{N_2(x)}},\phi_P^3\right\rangle\right|,
\]
where
\[
B_P(f,g):=\sum_{Q\in\vec{\mathbf{Q}}:Q_3\subset \tilde{P}_2}\frac{1}{|I_Q|^{1/2}}\langle f_1, \phi_{Q}^1\rangle\langle f_2, \phi_{Q}^2\rangle\phi_{Q}^3,
\]
provided the estimates are deduced in a way which does not depend on $\vec{\mathbf{P}}$ and $\vec{\mathbf{Q}}$.
\end{assumption}

\section{Restricted Weak-Type Interpolation}
In this chapter, we discuss the so-called restricted weak-type interpolation method.  This method is valid for general $n$-linear operators, but we state them here for our specialized case.

\begin{definition}
A tuple $\alpha=(\alpha_1,\alpha_2,\alpha_3,\alpha_4)$ is called admissible if 
\begin{enumerate}
\item $-\infty<\alpha_i<1$ for all $i=1,2,3,4$
\item $\sum \alpha_i=1$
\item At most one $\alpha_i<0$.
\end{enumerate}
We call an index $i$ good if $\alpha_i\ge 0$ and bad if $\alpha_i<0$.  A good tuple is an admissible tuple without a bad index.  A bad tuple is a tuple with a bad index.
\end{definition}
\begin{definition}We define the term \textbf{majorant} as follows.

\begin{enumerate}
\item If $\alpha$ and $\beta$ are good tuples and there exists a $j_0$ such that
\[
\alpha_j<\beta_j\textrm{ for all }j\neq j_0,
\]
then we say that $\beta$ is a majorant of $\alpha$ with index $j_0$.
\item If $\alpha$ or $\beta$ is a bad tuple, we assume that $j_0$ is the bad index (if they are both bad, this $j_0$ is the same for both).  In this case, we say that $\beta$ is a majorant of $\alpha$ with index $j_0$ if
\[
\alpha_j<\beta_j\textrm{ for all }j\neq j_0.
\]
\end{enumerate}
\end{definition}
\begin{definition}
Let $E,E'$ be sets of finite measure.  We say that $E'$ is a major subset of $E$ if $E'\subseteq E$ and $|E'|\ge \frac{1}{2}|E|$.
\end{definition}
\begin{definition}
If $E$ is a set of finite measure, we denote by $X(E)$ the space of functions supported on $E$ such that $\|f\|_\infty\le 1$.
\end{definition}
\begin{definition}
If $\alpha$ is an admissible tuple, we say that a $4$-linear form $\Lambda$ is of restricted weak-type\footnote{It is worth mentioning here that this is a slightly stronger definition of restricted weak type than others which appear in the literature, e.g. \cite{weaktypeinterpolation}.  That said, there is a much stronger interpolation theorem available for this variant.} $\alpha$ if for every sequence $E_1,E_2,E_3,E_4$ of subsets of $\mathbb{R}$ of finite measure, there exists a major subset $E_j'$ of $E_j$ for each bad index $j$ (there is at most one, though possibly none) such that
\[
\Lambda(f_1,f_2,f_3,f_4)\lesssim |E'|^\alpha,
\]
for all $f_i\in X(E_i)$, $i=1,2,3,4$, where we adopt the convention that $E'_i=E_i$ when $i$ is a good index, and
\[
|E'|^\alpha=|E_1'|^{\alpha_1}|E_2'|^{\alpha_2}|E_3'|^{\alpha_3}|E_4'|^{\alpha_4}.
\]
\end{definition}
\begin{definition}
Suppose that a 4-linear form $\Lambda$ is of restricted weak type $\alpha$ for some family of tuples $\alpha\in A$ which all have the same bad index $j_0$.  Suppose further that the same major subset $E_{j_0}'$ in the definition of restricted weak type can be used for all elements of $A$.  Then we say that $\Lambda$ is of uniformly restricted weak type.
\end{definition}
The basic idea here is that if $\Lambda(f_1,f_2,f_3,f_4)=\int T(f_1,f_2,f_3)f_4dx$, then a good tuple can be written as $(1/p_1,1/p_2,1/p_3,1/p_4)$ and corresponds to a standard H\"{o}lder type estimate for $T$, i.e. $L^{p_1}\times L^{p_2}\times L^{p_3}\rightarrow L^{p_4'}$.  If a tuple had bad index 4, say, then the target space of $T$, $L^{p_4'}$, is necessarily not a Banach space since $1/p_4'<1$.  Thus one cannot invoke immediately more standard interpolation results about mappings between Banach spaces.  See, for example, \cite{strichartz}.\footnote{It was quite a treat, years ago, to go looking for the original source of this result and to discover it was written by the author's friend and former REU mentor, Bob Strichartz.}

The following theorem guarantees that one can interpolate multilinear restricted weak-type estimates as one can with usual multilinear estimates, provided the interpolated tuple is a good tuple.
\begin{theorem}\label{goodtuples}
Let $\alpha^{(1)},...,\alpha^{(4)}$ be admissible tuples, and let $\alpha$ be a \textbf{good} tuple such that
\[
\alpha=\theta_1\alpha^{(1)}+...+\theta_4\alpha^{(4)},
\]
where $0<\theta_s< 1$ for $s=1,2,3,4$ and $\theta_1+...+\theta_4=1$.  Suppose that $\Lambda$ is of restricted weak type $\alpha^{(s)}$ for $s=1,2,3,4$.  Then $\Lambda$ is of restricted weak type $\alpha$.
\end{theorem}
\begin{proof}
Consider the quantities
\[
|\Lambda(f_1,f_2,f_3,f_4)|^{\theta_i}\lesssim \left(|E|^{\alpha^{(i)}}\right)^{\theta_i}
\]
and multiply them together.
\end{proof}
The following theorem says that at good tuples on the interior of a convex, open set where a 4-linear form is of restricted weak type, then it is of strong type on the interior of the set.
\begin{theorem}
Let $\alpha^{(1)},...,\alpha^{(4)}$ be tuples, and let $\alpha$ be a \textbf{good} tuple in the interior of the convex hull of $\alpha^{(1)},...,\alpha^{(4)}$.  Suppose that $\Lambda$ is of restricted weak-type $\alpha^{(s)}$ for $s=1,2,3,4$.  Then $\Lambda$ is of strong-type $\alpha$.
\end{theorem}
\begin{proof}
See \cite[Corollary 1, pp 383--384]{weaktypeinterpolation}.
\end{proof}

These previous theorems actually hold for a weaker definition of restricted weak-type.  They are not strong enough for our purposes because they require all the interpolated tuples to be good in order to produce estimates.  The following three lemmas are replacements.

\begin{lemma}\label{badtuples1}
Suppose that a 4-linear form $\Lambda$ is of uniformly restricted weak type $\alpha^{(s)}$ for $s=1,...,4$, where all bad indices, if they exist, have the same bad index.  Suppose that
\[
\alpha=\theta_1\alpha^{(1)}+...+\theta_4\alpha^{(4)},
\]
where $0<\theta_s< 1$ for $s=1,2,3,4$ and $\theta_1+...+\theta_4=1$.  Then $\Lambda$ is of uniform restricted weak type for $\{\alpha,\alpha^{(1)},\alpha^{(2)},\alpha^{(3)},\alpha^{(4)}\}$.  Thus $\Lambda$ is of uniform restricted weak type in the interior of the convex hull of the $\alpha^{(s)}$.
\end{lemma}
\begin{proof}
Consider the quantities
\[
|\Lambda(f_1,f_2,f_3,f_4)|^{\theta_i}\lesssim \left(|E|^{\alpha^{(i)}}\right)^{\theta_i}
\]
and multiply them together, using the uniformity in the major subset.
\end{proof}
\begin{lemma}\label{badtuples2}
Suppose that $\alpha^{(s)}$ is a collection of tuples which are either good or bad with a fixed bad index for which $\Lambda$ is of restricted weak type.  Let
\[
\alpha:=\theta_1\alpha^{(1)}+...+\theta_4\alpha^{(4)},
\]
where $0<\theta_s< 1$ for $s=1,2,3,4$ and $\theta_1+...+\theta_4=1$.  We assume that some $\alpha^{(j)}$ is a majorant of $\alpha$ with index $j_0$, where
\begin{enumerate}
\item if $\alpha$ is good then $j_0$ is an index for which $\alpha_{j_0}>0$.
\item if $\alpha$ is bad then $j_0$ is that index.
\end{enumerate}
Then one has that $\Lambda$ is of restricted weak type $\alpha$ as well.
\end{lemma}
\begin{proof}
See the appropriate appendix of \cite{camilbook2}.  It is also essentially \cite[Lemma 3.10]{interp}
\end{proof}

These two lemmas give one the ability to interpolate between restricted weak-type estimates.  However, we really want to be able to produce strong estimates for bad tuples.  This is accomplished through the following lemma, which is just a special case of \cite[Lemma 3.11]{interp}.

\begin{lemma}\label{restrictedweaktostrong}
Let $\alpha$ be a bad tuple with bad index 4.  Suppose that our 4-linear form $\Lambda(f_1,f_2,f_3,f_4)$ satisfies a restricted weak-type estimate in an open neighborhood of $\alpha$.  Then if $\alpha_i=1/p_i$ for $i=1,2,3$ with $1<p_1,p_2,p_3<\infty$ and $\alpha_4=1/p_4=1-(1/p_4')$ with $1\le p_4'<\infty$, we have
\[
\|T(f_2,f_2,f_3)\|_{p_4'}\le C\|f_1\|_{p_1}\|f_2\|_{p_2}\|f_3\|_{p_3}
\]
for all functions $f_i$ supported on a set of finite measure.
\end{lemma}
This lemma says that once one has a tiny open set worth of restricted weak-type estimates, one can get strong estimates on the interior for a class of functions like $C_c^\infty(\mathbb{R})$, which is enough to extend to strong boundedness of $T$ by the usual density arguments.

\section{Size and Energy Estimates}
\begin{notation}
For ease of writing, we will make the following notation:
\[
\tilde{\phi}_P^3:=\phi_P^31_{|I_P|\ge 2^{N_2(x)}}
\]
\end{notation}
We also recall the following:
\begin{notation}
Given a rank 1 family of tri-tiles $\vec{\mathbf{Q}}$, suppose that $Q\in\vec{\mathbf{Q}}$.  $Q$ is then made up of three tiles, each given by the product of a fixed interval $I_Q$ with a frequency interval, which we will denote $\omega_{Q_t}$, $t=1,2,3$.
\end{notation}

The model in question is given by
\[
\sum_{P\in\vec{\mathbf{P}}}\frac{1}{|I_P|^{1/2}}\langle f_3, \phi_P^1\rangle\left\langle B_P(f_1,f_2), \phi_P^2\right\rangle\left\langle f_4,\tilde{\phi}_P^3\right\rangle,
\]
where
\[
B_P(f_1,f_2):=\sum_{Q\in\vec{\mathbf{Q}}:\omega_{Q_3}\subset \omega_{\tilde{P}_2}}\frac{1}{|I_Q|^{1/2}}\langle f_1, \phi_{Q}^1\rangle\langle f_2, \phi_{Q}^2\rangle\phi_{Q}^3,
\]
where $\vec{\mathbf{P}}$ and $\vec{\mathbf{Q}}$ are sparse, finite, rank 1 families of tri-tiles.

Following the standard multilinear harmonic analysis approach, as in \cite{DTT}, \cite{MTTBiest2}, \cite{camilbook}, and many others, we wish to discuss sizes, which will require the notion of a tree.
\begin{definition}
For any $t\in\{1,2,3\}$ and a tri-tile $\vec{P}_T\in\vec{\mathbf{P}}$, we define a $j$-tree with top $\vec{P}_T$ to be a collection of tri-tiles $T\subset\vec{\mathbf{P}}$ such that
\[
P_j\le P_{T,t}\textrm{ for all }\vec{P}\in T,
\]
where $P_{T,t}$ is the $t$-component of $\vec{P}_T$.  We will write $I_T$ and $\omega_{T,t}$ for $I_{\vec{P}_T}$ and $\omega_{\vec{P}_{T,t}}$, respectively.  We say that $T$ is a tree if it is a $t$-tree for some $1\le t\le 3$.
\end{definition}
It is worth remarking that a tree does not necessarily have to contain its top.
\begin{definition}
We will say that a tree $T$ is $t$-lacunary if it is a $t'$-tree for some $t\neq t'$.
\end{definition}
\begin{definition}
Let $t\in\{1,2,3\}$.  Two trees $T$ and $T'$ are said to be strongly $i$-disjoint if
\begin{enumerate}
\item $P_i\neq P_i'$ for all $\vec{P}\in T$ and $\vec{P}'\in T'$.
\item Whenever $\vec{P}\in T$, $\vec{P}'\in T'$, are such that $2\omega_{P_i}\cap 2\omega_{P_i'}\neq\emptyset$, then one has $I_{\vec{P}'}\cap I_T=\emptyset$, and similarly with $T$ and $T'$ reversed.
\end{enumerate}
\end{definition}

\subsection{Sizes}
\begin{definition}{\textbf{Sizes.}}
Suppose that $\vec{\mathbf{P}}$ is a finite collection of tri-tiles and $t\in\{1,2,3\}$.  Suppose also that $(a_{P_j})_{\vec{P}\in\vec{\mathbf{P}}}$ is a sequence of complex numbers.  Here one really should think of $a_{P_j}$ as being a sequence ``living'' on the tiles $P_j$ rather than the full tri-tile.
\[
\sizee_j(\left(a_{P_j})_{\vec{P}\in\vec{\mathbf{P}}}\right):=\sup_{T\subset\vec{\mathbf{P}}}\left(\frac{1}{|I_T|}\sum_{\vec{P}\in T}|a_{P_j}|^2\right)^{1/2},
\]
where the $T$ in the supremum ranges over all trees in $\vec{\mathbf{P}}$ which are $i$-trees for some $i\neq j$.  In other words, the supremum ranges over all trees which are $j$-lacunary.
\end{definition}
The above definitions work for general sequences, but for our purposes, we should keep in mind that the sequences we are interested in are
\begin{enumerate}
\item $a_{P_1}=\langle f_1,\phi_P^1\rangle$
\item $a_{P_2}=\left\langle B_P(f,g), \phi_P^2\right\rangle$
\item $a_{P_3}=\left\langle f_4,\tilde{\phi}_P^3\right\rangle$
\end{enumerate}

The heuristic meaning of these sizes is that the size of a sequence is a measure the extent to which it can concentrate on a single tree.  It should be thought of as a phase-space variant of the BMO norm.  Indeed, one has a relevant variant of the John-Nirenberg inequality:
\begin{proposition} If $\mathcal{I}$ is a finite family of dyadic intervals, $r$ is any positive real number and $(a_I)_{I\in\mathcal{I}}$, then define $\|(a_I)_I\|_{BMO(r)}$ by
\[
\|(a_I)_I\|_{BMO(r)}:=\sup_{I_0\in\mathcal{I}}\frac{1}{|I_0|^{\frac{1}{r}}}\left\|\left(\sum_{I\subseteq I_0}\frac{|a_I|^2}{|I|}\chi_I(x)\right)^{1/2}\right\|_r.
\]
Then if $0<p<q<\infty$,
\[
\|(a_I)_I\|_{BMO(p)}\sim \|(a_I)_I\|_{BMO(q)}.
\]
\end{proposition}
\begin{proof}
See the appropriate section of Chapter 2 of \cite{camilbook}.
\end{proof}
The sizes defined above roughly correspond to this $BMO(r)$ norm when $r=2$.  We state several lemmas which will be used to estimate our model.  Since we will be using restricted weak-type interpolation (explained later on), we should recall a previous definition:
\begin{definition}
Suppose that $E$ is a set of finite measure.  We define the space $X(E)$ to denote the space of all functions $f$ supported on $E$ with $\|f\|_\infty\le 1$.
\end{definition}
The following three lemmas are the size estimates we require:
\begin{lemma}\label{size1lemma}
Let $E_1$ be a set of finite measure, let $f_3$ be in $X(E_3)$, and let $\vec{\mathbf{P}}$ be a finite collection of tri-tiles.  Then one has
\[
\sizee_1((\langle f_3,\phi_P^1\rangle)_{\vec{P}\in\vec{\mathbf{P}}})\lesssim \sup_{\vec{P}\in\vec{\mathbf{P}}}\frac{\int_{E_3}\tilde{\chi}_{I_{\vec{P}}}^M}{|I_{\vec{P}}|},
\]
for all $M>0$, with the implicit constant depending on $M$.
\end{lemma}
\begin{proof}
See Lemma 6.8 in \cite{MTTBiest2}.
\end{proof}
\begin{lemma}
Let $E_1,E_2$ be sets of finite measure, let $f_1,f_2$ be in $X(E_1)$ and $X(E_2)$, respectively, and let $\vec{\mathbf{P}}$ be a finite collection of tri-tiles.  Let
\[
(a_{P_2})_{\vec{P}\in\vec{\mathbf{P}}}:=\left(\sum_{Q\in\vec{\mathbf{Q}}:\omega_{Q_3}\subset \omega_{\tilde{P}_2}}\frac{1}{|I_Q|^{1/2}}\langle f_1, \phi_{Q}^1\rangle\langle f_2, \phi_{Q}^2\rangle\phi_{Q}^3,\phi_P^2\right)_{\vec{P}\in\vec{\mathbf{P}}}.
\]
Then one has
\[
\sizee_2\left((a_{P_2})_{\vec{P}\in\vec{\mathbf{P}}}\right)\lesssim \sup_{\vec{P}\in\vec{\mathbf{P}}}\left(\frac{\int_{E_1}\tilde{\chi}_{I_{\vec{P}}}^M}{|I_{\vec{P}}|}\right)^{\theta}\left(\frac{\int_{E_2}\tilde{\chi}_{I_{\vec{P}}}^M}{|I_{\vec{P}}|}\right)^{1-\theta},
\]
for any $0<\theta<1$ and all $M>0$, with the implicit constant depending on $\theta,M$.
\end{lemma}
\begin{proof}
See Lemma 9.1 in \cite{MTTBiest2}.
\end{proof}
Finally, we state Theorem 6.2 from \cite{DTT}, modulo some trivial changes of notation.  This statement is of the same flavor as the previous two, modulo a technical limitation which seemingly can be done away with, but we leave it in for ease of use:
\begin{lemma}
Let $E_4$ be a set of finite measure, let $f_4$ be in $X(E_4)$, and let $\vec{\mathbf{P}}$ be a finite collection of tri-tiles.  We define
\[
\mathcal{I}_{\vec{\mathbf{P}}}:=\{I\textrm{ dyadic }:I_{\vec{P}}\subseteq I\subseteq I_{\vec{P}'}\textrm{ for some }\vec{P},\vec{P}'\in\vec{\mathbf{P}}\}
\]
Then one has
\[
\sizee_3((\langle f_4,\tilde{\phi}_P^3\rangle)_{\vec{P}\in\vec{\mathbf{P}}})\lesssim \sup_{I\in\mathcal{I}_{\vec{\mathbf{P}}}}\frac{\int_{E_4}\tilde{\chi}_{I}^M}{|I|},
\]
for all $M>0$, with the implicit constant depending on $M$.
\end{lemma}
\begin{remark}
We believe the proof in \cite{DTT} of the above lemma, unfortunately, has a non-trivial error.  We believe this error may be repairable, but rather than merely citing their result, we provide the reader with an alternative lemma which, in any case, can be used as a replacement in both the present work and in \cite{DTT}.
\end{remark}
\begin{lemma}
Let $E_4$ be a set of finite measure, let $f_4$ be in $X(E_4)$, and let $\vec{\mathbf{P}}$ be a finite collection of tri-tiles.  Then for every $q\in(1,\infty)$, one has
\[
\sizee_3((\langle f_4,\tilde{\phi}_P^3\rangle)_{\vec{P}\in\vec{\mathbf{P}}})\lesssim \sup_{\mathbf{T}\subset\mathbf{\vec{P}}}\frac{1}{|I_T|^{1/q}}\|f\tilde{\chi}_{I_T}^M\|_q\lesssim \sup_{\{I:I=I_{{\vec{P}}}, \vec{P}\in\mathbf{\vec{P}}\}}\frac{\int_{E_4}\tilde{\chi}_{I}^{qM}}{|I|^q},
\]
for all $M>0$, with the implicit constant depending on $M$.
\end{lemma}
\begin{proof}
We ignore the superfluous subscripts, i.e. let $f=f_4$ and suppose that $\mathbf{T}$ is any lacunary tree with top $I_T$ and let a typical tile in $\mathbf{T}$ be denoted by $Q$.  The quantity we wish to estimate is given by
\[
\frac{1}{|I_T|^{1/2}}\left(\sum_{Q\in\mathbf{T}}|\langle f,\tilde{\phi}_{Q}\rangle|^2\right)^{1/2}.
\]
We now linearize this expression by dualizing with an arbitrary sequence $(a_Q)_{Q\in\mathbf{T}}$ so that $\|(a_Q)\|_{\ell^2}\le 1$ to produce
\begin{align*}
\frac{1}{|I_T|^{1/2}}\left(\sum_{Q\in\mathbf{T}}|\langle f,\tilde{\phi}_{Q}\rangle|^2\right)^{1/2}&=\frac{1}{|I_T|^{1/2}}\left(\sum_{Q\in\mathbf{T}}a_Q\langle f,\tilde{\phi}_{Q}\rangle\right)\\
&=\frac{1}{|I_T|^{1/2}}\left\langle f,\sum_{Q\in\mathbf{T}}a_Q\tilde{\phi}_Q\right\rangle,
\end{align*}
where we have ignored a complex conjugation over the $a_Q$, which is completely harmless.  If we let $I_{T,n}$ denote the translation of $I_T$ by $n|I_T|$ units, so we may write $f=\sum_{n\in\Z}f\chi_{I_{T,n}}$.  For $|n|\ge 2$, we may perform crude estimates as in the beginning of the proof of Theorem 6.2 from \cite{DTT}, so we ignore these terms and assume that $f$ is supported in $3I_T$.  In such a case, the above expression is
\[
\left\langle f,\sum_{Q\in\mathbf{T}}a_Q\tilde{\phi}_Q\right\rangle=\int_{3I_T}f\sum_{Q\in\mathbf{T}}a_Q\tilde{\phi}_Qdx,
\]
where he have again ignored a complex conjugation in the inner product.  For any $p,q$ which are conjugate exponents, we invoke H\"{o}lder's inequality to get
\[
\left|\int_{3I_T}f\sum_{Q\in\mathbf{T}}a_Q\tilde{\phi}_Qdx\right|\le \left\|\sum_{Q\in\mathbf{T}}a_Q\tilde{\phi}_Q\right\|_{L^p(3I_T)}\|f\|_{L^q(I_T)}
\]  
Using standard estimates, as in \cite{DTT}, one has a pointwise estimate
\[
\left|\sum_{Q\in\mathbf{T}}a_Q\tilde{\phi}_Q\right|\le M\left(\sum_{Q\in\mathbf{T}}a_Q\phi_Q\right),
\]
where $M$ is the Hardy--Littlewood maximal operator.  Thus by classical theorems about $M$, it suffices to estimate
\[
\left\|\sum_{Q\in\mathbf{T}}a_Q\phi_Q\right\|_{L^p(3I_T)}.
\]
Now dualize with a function $g$ with $\|g\|_{L^{p'}(3I_T)}\le 1$.  Then the above can be majorized (again ignoring complex conjugation) by
\begin{align*}
\left|\int_{3I_T}\sum_{Q\in\mathbf{T}}a_Q\phi_Q gdx\right|&=\left|\sum_{Q\in\mathbf{T}}a_Q\langle g,\phi_Q \rangle\right|=\left|\sum_{Q\in\mathbf{T}}\int_{I_Q}\frac{a_Q}{|I_Q|^{1/2}}\frac{\langle g,\phi_Q\rangle}{|I_Q|^{1/2}}\chi_{I_Q}dx\right|\\
&=\left|\int_{3I_T}\sum_{Q\in\mathbf{T}}\frac{a_Q}{|I_Q|^{1/2}}\frac{\langle g,\phi_Q\rangle}{|I_Q|^{1/2}}\chi_{I_Q}dx\right|\\
&\le \int_{3I_T}\left(\sum_{Q\in\mathbf{T}}\frac{|a_Q|^2}{|I_Q|}\chi_{I_Q}\right)^{1/2}\left(\sum_{Q\in\mathbf{T}}\frac{|\langle g,\phi_Q\rangle|^2}{|I_Q|}\chi_{I_Q}\right)^{1/2}dx\\
&\le \left\|\left(\sum_{Q\in\mathbf{T}}\frac{|a_Q|^2}{|I_Q|}\chi_{I_Q}\right)^{1/2}\right\|_{L^p(3I_T)}\left\|\left(\sum_{Q\in\mathbf{T}}\frac{|\langle g,\phi_Q\rangle|^2}{|I_Q|}\chi_{I_Q}\right)^{1/2}\right\|_{L^{p'}(3I_T)}
\end{align*}
The second factor in this expression is essentially a Littlewood--Paley square function, owing to the fact that $\mathbf{T}$ is actually a lacunary tree; thus it is bounded by $\|g\|_{L^{p'}(3I_T)}\le 1$.  The second term is then controlled by $|I_T|^{1/p}\|(a_Q)\|_{BMO(p)}$, which, by the theorem stated earlier in this section, is comparable to $|I_T|^{1/p}\|(a_Q)\|_{BMO(2)}$, which is majorized by
\[
\frac{|I_T|^{1/p}}{|I_T|^{1/2}}\|(a_Q)\|_{\ell^2}\le\frac{1}{|I_T|^{1/2-1/p}}.
\]
Finally, this results in the estimate
\[
\frac{1}{|I_T|^{1/2}}\left(\sum_{Q\in\mathbf{T}}|\langle f,\tilde{\phi}_{Q}\rangle|^2\right)\lesssim\frac{1}{|I_T|^{1/2}}\|f\|_{L^q(3I_T)}\frac{1}{|I_T|^{1/2-1/p}}=\frac{1}{|I_T|^{1/q}}\|f\|_{L^q(3I_T)}.
\]
From this, the general estimate follows.
\end{proof}
\begin{remark}
The fact that this is an $L^q$ average rather than an $L^1$ average, i.e. the presence of the $q>1$ rather than $q=1$, is completely harmless --- in what follows, there is always small loss in the exponents with the caveat that it can be made arbitrarily small.  By taking $q$ to be very close to $1$, we can still make this loss arbitrarily small, and so the argument using this lemma in lieu of the lemma from \cite{DTT} is essentially unchanged.
\end{remark}
\subsection{Energies}
We define the energies in this case as follows.  The $1$- and $2$-energies are modified somewhat from the ``standard'' energies.
\begin{definition}
If $t=1$ or $t=2$
\[
\energye_t ((a_{P_t})_{\vec{P}\in\vec{\mathbf{P}}}):=\sup_{n\in\mathbb{Z}}\sup_{\mathcal{F}}2^n\left(\sum_{T\in\mathcal{F}}|I_T|\right)^{1/2},
\]
where the second supremum ranges over all forests $\mathcal{F}$ consisting of strongly $t$-disjoint $t$-lacunary trees in $\vec{\mathbf{P}}$ such that
\[
\left(\sum_{\vec{P}\in T}|a_{P_t}|^2\right)^{1/2}\ge 2^n|I_T|^{1/2}
\]
for all $T\in\mathcal{F}$ and
\[
\left(\sum_{\vec{P}\in T'}|a_{P_t}|^2\right)^{1/2}\le 2^{n+1}|I_{T'}|^{1/2}
\]
for all sub-trees $T'\subset T\in\mathcal{F}$.
\end{definition}
And here are the relevant estimates for the 1- and 2-energies:
\begin{lemma}
Let $f_3$ be a function in $X(E_3)$ and $\vec{\mathbf{P}}$ a finite collection of tri-tiles.  Then
\[
\energye_1((\langle f_3,\phi_P^1\rangle)_{\vec{P}\in\vec{\mathbf{P}}})\le |E_3|^{1/2}.
\]
\end{lemma}
\begin{proof}
See Lemma 6.7 from \cite{MTTBiest2}.
\end{proof}
\begin{lemma}
Suppose $E_1,E_2$ be sets of finite measure and $f_1,f_2$ functions with $f_1\in X(E_1)$ and $f_2\in X(E_2)$.  Let
\[
(a_{P_2})_{\vec{P}\in\vec{\mathbf{P}}}:=\left(\sum_{Q\in\vec{\mathbf{Q}}:\omega_{Q_3}\subset \omega_{\tilde{P}_2}}\frac{1}{|I_Q|^{1/2}}\langle f_1, \phi_{Q}^1\rangle\langle f_2, \phi_{Q}^2\rangle\phi_{Q}^3,\phi_P^2\right)_{\vec{P}\in\vec{\mathbf{P}}}.
\]
Then
\[
\energye_2\left((a_{P_3})_{\vec{P}\in\vec{\mathbf{P}}}\right)\lesssim \left(|E_1|^{1/2}\sup_{\vec{Q}\in\vec{\mathbf{Q}}}\frac{\int_{E_2}\tilde{\chi}^M_{I_{\vec{Q}}}}{|I_{\vec{Q}}|}\right)^{\theta}\left(|E_2|^{1/2}\sup_{\vec{Q}\in\vec{\mathbf{Q}}}\frac{\int_{E_1}\tilde{\chi}^M_{I_{\vec{Q}}}}{|I_{\vec{Q}}|}\right)^{1-\theta}.
\]
for any $0<\theta<1$
\end{lemma}
\begin{proof}
See Lemma 9.2 in \cite{MTTBiest2}, modulo some obvious changes of notation.
\end{proof}
There is not exactly a 3-energy.  However, we have the following replacement which is of a similar flavor:
\begin{lemma}\label{dttenergy}
Let $\mu>0$.  Suppose that $\mathcal{F}$ is a forest of strongly $3$-disjoint, $3$-lacunary trees.  Suppose further that $f\in X(E)$ is such that
\[
\left(\sum_{\vec{P}\in T}|\langle f_4,\tilde{\phi}_P^3\rangle|^2\right)^{1/2}\ge 2^n|I_T|^{1/2}
\]
for all $T\in\mathcal{F}$ and
\[
\left(\sum_{\vec{P}\in T'}|\langle f_4,\tilde{\phi}_P^3\rangle|^2\right)^{1/2}\le 2^{n+1}|I_{T'}|^{1/2}
\]
for all sub-trees $T'\subseteq T\in\mathcal{F}$.  Then
\[
\left(\sum_{T\in\mathcal{F}}|I_T|\right)^{1/2}\lesssim |E_4|^{1/2}2^{-n}\left(2^{-n}|E_4|^{-1/2}\right)^{1/\mu},
\]
where the implicit constant depends on $\mu$.
\end{lemma}
\begin{proof}
See Lemma 9.2 in \cite{DTT}.  This is the primary lemma of Demeter, Tao, and Thiele's paper and requires roughly 20 pages of computations.  The main idea is the following.  Let 
\[
N_\mathcal{F}:=\sum_{T\in\mathcal{F}}1_{I_T},
\]
and suppose that $I_0$ is any interval which contains the support of $N_\mathcal{F}$.  With a lot of hard work and the help of a theorem of Rademacher--Menshov and a lemma of Bourgain, one can establish the following for any $\mu>0$:
\[
\sum_{\vec{P}\in\cup_{T\in\mathcal{F}}T} |\langle f_4,\phi_P^31_{|I_P|>2^{N(x)}}\rangle|^2\lesssim \|N_\mathcal{F}\|_\infty^{1/\mu}\int |f_4|^2\chi_{I_0}^{10}.
\]
More precisely, one shows that one loses at most a small power of the logarithm of $\|N_\mathcal{F}\|_\infty$.  The two hypotheses guarantee that our estimate is still preserved after restricting to subtrees, which, it turns out, is precisely enough to get the desired conclusion.
\end{proof}
The factor $(2^{-n})^{1/\mu}$ is essentially technical and can basically be ignored; however, its presence bars one from taking the desired supremum over $n$ in the definitions of 1- and 2-energies.  That said, we can use this lemma to establish the following:
\begin{lemma}\label{12partitionlemma}
Let $\vec{\mathbf{P}}$ be a finite collection of multitiles.  Let $\mu>0$.  Then after discarding tiles $\vec{P}$ such that $\langle f_4,\tilde{\phi}_P^3\rangle=0$, there exists a partition,
\[
\vec{\mathbf{P}}=\bigcup_{n:2^n\le \sizee_3((a_{P_3})_{\vec{P}\in\vec{\mathbf{P}}})}\bigcup_{T\in\mathcal{F}^{n,3}}T,
\]
where $\mathcal{F}^{n,3}$ is a collection of trees such that $\sizee_3(T)\le 2^{m+1}$ and
\[
\sum_{T\in\mathcal{F}^{n,3}}|I_T|\lesssim |E_4|2^{-2n}\left(2^{-n}|E_4|^{-1/2}\right)^{2/\mu}
\]
\end{lemma}
\begin{proof}
See Corollary 6.4 in \cite{DTT}.
\end{proof}

One gets nearly identical partition results for the $P_2$ and $P_1$ sequences using the energy results described for them, except that there is no presence of $(2^{-n})^{1/\mu}$ in these cases.

\section{Estimating the Four-Linear Form}
The application of the sizes, energies, and weak-type interpolation is fairly standard (for example, as in the  article which inspired the present work, \cite{DTT}), but we reproduce the procedure here.  

We now state a basic lemma.  It essentially comes from the intuition that one can estimate
\[
\left|\sum_n a_nb_nc_n\right|\le \|a_n\|_{\ell^2}\|b_n\|_{\ell^2}\|c_n\|_{\ell^\infty}
\]
\begin{lemma}\label{basiclemma}
Suppose that $T$ is a $t$-tree contained in $\vec{\mathbf{P}}$.  This means it is a $t'$-lacunary tree for $t'\neq t$.  As before, let
\begin{enumerate}
\item $a_{P_1}=\langle f_3,\phi_P^1\rangle$
\item $a_{P_2}=\left\langle B_P(f_1,f_2), \phi_P^2\right\rangle$
\item $a_{P_3}=\left\langle f_4,\tilde{\phi}_P^3\right\rangle$
\end{enumerate}
Then
\begin{align*}
\left|\sum_{\vec{P}\in T}\right. &\left|\frac{1}{|I_P|^{1/2}}\langle f_3, \phi_P^1\rangle\left\langle B_P(f_1,f_2), \phi_P^2\right\rangle\left\langle f_4,\tilde{\phi}_P^3\right\rangle\right| \lesssim \\
\sum_{\vec{P}\in T}&\left|\frac{1}{|I_P|^{1/2}}\langle f_3, \phi_P^1\rangle\left\langle B_P(f_1,f_2), \phi_P^2\right\rangle\left\langle f_4,\tilde{\phi}_P^3\right\rangle\right| \lesssim \\
&|I_T|\sizee_1((a_{P_1})_{\vec{P}\in T})\cdot\sizee_2((a_{P_2})_{\vec{P}\in T})\cdot\sizee_3((a_{P_3})_{\vec{P}\in T}).
\end{align*}
\end{lemma}
\begin{proof}
By the definition of size,
\[
\left(\sum_{\vec{P}\in T}|a_{P_{t'}}|^2\right)^{1/2}\lesssim |I_T|^{1/2}\size_{t'}((a_{P_{t'}})_{\vec{P}\in T}),
\]
for each $t'\neq t$.  For $t$, one has that a single tile is a tree, and so
\[
|a_{P_t}|\lesssim |I_{P}|^{1/2}\size_t((a_{P_t})_{\vec{P}\in T}).
\]
The claim then follows by the $\ell^2\times \ell^2\times \ell^\infty$ version of the H\"{o}lder inequality (basically just Cauchy-Schwarz).
\end{proof}
Supposing that $f_t\in X(E_t)$, this means it is enough (by restricted weak-type interpolation) to break up $\vec{\mathbf{P}}$ into trees $T$ where one can produce the estimate
\[
\sum_T |I_T|\size_1((a_{P_1})_{\vec{P}\in T})\cdot\size_2((a_{P_2})_{\vec{P}\in T})\cdot\size_3((a_{P_3})_{\vec{P}\in T})\lesssim |E_1|^{\alpha_1}|E_2|^{\alpha_2}|E_3|^{\alpha_3}|E_4|^{\alpha_4},
\]
for an admissible tuple $\alpha=(\alpha_1,\alpha_2,\alpha_3,\alpha_4)$, where $\alpha_1+\alpha_2+\alpha_3+\alpha_4=1$.\footnote{This last condition is clearly required since the operator in question behaves something like a pointwise product, and thus should satisfy H\"{o}lder-type estimates.}

As per the restricted weak-type interpolation theorems, we are allowed to remove a certain subset from the $E_n$ corresponding to a bad index (in the event that a bad index exists, or to any index in the event that no bad index exists).  The indices 1 and 2 are to be handled differently from the indices 3 and 4: the functions $f_1,f_2$ are mixed together, and so their will have to be treated in a slightly different way than those for $f_3,f_4$.  However, there is no difference between the methods used to handle 3 or 4.
\subsection{Estimates when $3$ or $4$ is the bad index}
  We will describe in detail how to do this for index 4 being bad; the index 3 case can be done completely analogously.

We will now define the exceptional set.  For $C>0$, define $\Omega_C$ as
\[
\Omega_C:=\bigcup_{i=1}^4\{x:M(1_{E_i})\ge C|E_i|/|E_4|\},
\]
where $M$ is the usual Hardy--Littlewood maximal operator.  For sufficiently large $C$, we can guarantee that $|E_4/\Omega_C|\ge \frac{1}{2}|E_4|$.  Let $E_4'$ be $E_4\backslash\Omega_C$ for such a $C$.

Suppose that $f_1\in X(E_1),f_2\in X(E_2),f_3\in X(E_3),f_4\in X(E_4')$, and let $\vec{\mathbf{P}}$ be a finite rank 1 collection of tri-tiles.  We partition $\vec{\mathbf{P}}$ as follows: let $\vec{\mathbf{P}}_{l}$ be the collection of tri-tiles $\vec{P}$ such that $I_{\vec{P}}$ satisfies
\[
2^l\le 1+\frac{\textrm{dist}(I_{\vec{P}},\mathbb{R}/\Omega)}{|I_{\vec{P}}|} \le 2^{l+1}.
\]
We will then have to sum over $l$.  We shall find that we get an exponential gain of $2^{-l}$, so this will not be an issue.  Observe that, from our size estimates that for such collections of tiles,
\[
\size_1(a_{P_1})_{\vec{P}\in\vec{\mathbf{P}}})\lesssim \frac{|E_3|}{|E_4|}2^l
\]
\[
\size_2(a_{P_2})_{\vec{P}\in\vec{\mathbf{P}}})\lesssim \frac{|E_1|^\theta|E_2|^{1-\theta}}{|E_4|}2^l
\]
and
\[
\size_3(a_{P_3})_{\vec{P}\in\vec{\mathbf{P}}})\lesssim 2^{(1-M)l},
\]
where $M$ is the exponent from the definition of adaptedness to a tile.

Now, using Lemma \ref{12partitionlemma} (and the appropriate analogues for $P_1$ and $P_2$), generate families $\mathcal{F}^{n,1}$, $\mathcal{F}^{n,2}$, and $\mathcal{F}^{n,3}$.  After discarding tiles with $\langle f,\phi_P^t\rangle$, say, to zero, one can perform the partition,
\[
\vec{\mathbf{P}}^{l}=\bigcup_{m_1,m_2,m_3}\mathbf{S}^{m_1}\cap\mathbf{S}^{m_2}\cap\mathbf{S}^{m_3},
\]
where $\mathbf{S}^{m_t}:=\bigcup_{T\in\mathcal{F}^{m,t}}T$ and we assume implicitly that
\[
2^{m_t}\le\size_t((a_{P_t})_{\vec{P}\in\vec{\mathbf{P}}}).
\]
One can further partition,
\[
\vec{\mathbf{P}}^l=\bigcup_{j=1}^3\bigcup_{m_1,m_2,m_3:m_j=\max\{m_1,m_2,m_3\}}\bigcup_{T\in\mathcal{F}^{m_j,j}}(T\cap\mathbf{S}^{m_1}\cap\mathbf{S}^{m_2}\cap\mathbf{S}^{m_3}).
\]
Losing a factor of $3$ in the estimates, we may drop the union over $j$ and assume that
\[
\vec{\mathbf{P}}^l=\bigcup_{m_1,m_2,m_3:m_j=\max\{m_1,m_2,m_3\}}\bigcup_{T\in\mathcal{F}^{m_j,j}}(T\cap\mathbf{S}^{m_1}\cap\mathbf{S}^{m_2}\cap\mathbf{S}^{m_3}).
\]
It is worth observing that $T\cap\mathbf{S}^{m_1}\cap\mathbf{S}^{m_2}\cap\mathbf{S}^{m_3}$ is still a tree with the same top as $T$ and, by the sub-tree properties of the partition from Lemma \ref{12partitionlemma}, we have that its size is at most $2^{m_j+1}$.  Thus we must finally verify that
\[
\sum_{m_1,m_2,m_3:m_j=\max\{m_1,m_2,m_3\}}\sum_{T\in\mathcal{F}^{m_j,j}}|I_T|2^{m_1+m_2+m_3}\lesssim 2^{-l}|E_1|^{\alpha_1}|E_2|^{\alpha_2}|E_3|^{\alpha_3}|E_4|^{\alpha_4}.
\]
Suppose that $a_1+a_2+a_3=1=\frac{1-a_1}{2}+\frac{1-a_2}{2}+\frac{1-a_3}{2}$ for $0\le a_1,a_2,a_3\le1$.  Let $0<\theta<1$.  By Lemma \ref{12partitionlemma} and its two variants, we have that
\begin{align}\label{tempytemperton}
&\sum_{T\in\mathcal{F}^{m_j,j}}|I_T|\lesssim 2^{-2m_j}\left(|E_3|\right)^{\frac{1-a_1}{2}}\left(|E_1|^\theta|E_2|^{1-\theta}\right)^{\frac{1-a_2}{2}}\\
&\times\left(|E_4|\left(2^{-m_j}|E_4|^{-1/2}\right)^{2/\mu}\right)^{\frac{1-a_3}{2}}
\end{align}

Since we assumed implicitly that 
\[
2^{m_t}\le \size_t((a_{P_t})_{\vec{P}\in\vec{\mathbf{P}}}),
\]
we have that, for the same $a_1,a_2,a_3$:
\begin{align*}
2^{m_1+m_2+m_3}&=2^{m_1(1-a_1)+m_2(1-a_2)+m_3(1-a_3)}2^{m_1a_1+m_2a_2+m_3a_3}\\
&\le 2^{m_j}\prod_{i\neq j}\size_i((a_{P_i})_{\vec{P}\in\vec{\mathbf{P}}})^{a_i}2^{m_i(1-a_i)}
\end{align*}
Thus by summing up the geometric sums over $m_i$, which cap out at $m_j$, one has
\begin{align*}
&\sum_{m_1,m_2,m_3:m_j=\max\{m_1,m_2,m_3\}}\sum_{T\in\mathcal{F}^{m_j,j}}|I_T|2^{m_1+m_2+m_3}\lesssim\\
& \prod_{i\neq j}\size_i((a_{P_i})_{\vec{P}\in\vec{\mathbf{P}}})^{a_i}\sum_{m_j}2^{m_j}\left(\prod_{i\neq j}2^{m_j(1-a_i)}\right)\sum_{T\in\mathcal{F}^{m_j,j}}|I_T|
\end{align*}
Now, plugging in (\ref{tempytemperton}), and summing over the final geometric series and carefully doing some arithmetic on the exponents, one can majorize the previous expression by
\begin{align*}
&\left(\prod_{i=1}^3 \size_i((a_{P_i})_{\vec{P}\in\vec{\mathbf{P}}})^{a_i}\right) \left(\size_j((a_{P_j})_{\vec{P}\in\vec{\mathbf{P}}})\right)^{-(1-a_3)/\mu}\\
&\times|E_3|^{(1-a_1)/2}|E_1|^{\theta(1-a_2)/2}|E_2|^{(1-\theta)(1-a_2)/2}|E_4|^{(1-a_3)/2}|E_4|^{-(1-a_3)/\mu}|E_4|^{-1}.
\end{align*}
Let $(1-a_3)/\mu=\epsilon$.  Observe that the presence of the $\size_j^{-\epsilon}$ term is harmless except that it effectively changes the factor of $|E_4|^{-\epsilon}$ to $|E_3|^{-\epsilon}$ if $j=1$, and so on.  This can be remedied quite easily.  Supposing that $j=1$, pick $\alpha_1'=\alpha_1+\epsilon$ (which is ok for ``most'' choices of $\alpha_1$ since $\mu$ can be taken very large) and making the appropriate change $\alpha_3'=\alpha_1-\epsilon$.  Thus the $-\epsilon$ can always be pushed onto $E_4$.  But the key point is that one gets a weak-type estimate for all $\epsilon$, so one can get estimates arbitrarily close to $\epsilon=0$.  Thus we ignore this technicality.  We can thus majorize the previous expression by quantities arbitrarily close to
\begin{align*}
&|E_3|^{(1+a_1)/2}|E_1|^{\theta(1+a_2)/2}|E_2|^{(1-\theta)(1+a_2)/2}|E_4|^{(1+a_3)/2}|E_4|^{-1}\\
=&|E_3|^{(1+a_1)/2}|E_1|^{\theta(1+a_2)/2}|E_2|^{(1-\theta)(1+a_2)/2}|E_4|^{(-1+a_3)/2}
\end{align*}
whenever $0<a_1,a_2,a_3<1$ with $a_1+a_2+a_3=1$ and $0<\theta<1$.  All the associated tuples are admissible tuples, and hence our 4-linear form $\Lambda$ is of restricted weak-type for all such $\alpha,\theta$ pairs.  If one picks:
\begin{enumerate}
\item $a_1=2\alpha_3-1$,
\item $a_3=2\alpha_4+1$,
\item $a_2=2(\alpha_1+\alpha_2)-1$, and
\item $\theta=\alpha_1/(\alpha_1+\alpha_2)$,
\end{enumerate}
then the previous estimate becomes
\[
|E_1|^{\alpha_1}|E_2|^{\alpha_2}|E_3|^{\alpha_3}|E_4|^{\alpha_4},
\]
Of course the sum of the exponents is then $1$, and hence our 4-linear form is of restricted weak-type $\alpha$ whenever $\frac{1}{2}<\alpha_3<1$, $-\frac{1}{2}<\alpha_4<0$, $0< \alpha_1,\alpha_2< 1$, $\frac{1}{2}<\alpha_1+\alpha_2<1$, and $\alpha_1+\alpha_2+\alpha_3+\alpha_4=1$.  In particular, one gets a restricted weak-type estimate for 4-tuples arbitrarily close to
\[
\begin{array}{cccccc} 
(1,0,\frac{1}{2},-\frac{1}{2}) & (0,1,\frac{1}{2},-\frac{1}{2}) & (\frac{1}{2},0,1,-\frac{1}{2}) & (0,\frac{1}{2},1,-\frac{1}{2}).  \\
\end{array}
\]
One can do precisely the same analysis for $3$ being the bad index to get restricted weak-type estimates
\[
\begin{array}{cccccc} 
(1,0,-\frac{1}{2},\frac{1}{2}) & ( 1,0,-\frac{1}{2},\frac{1}{2}) & ( \frac{1}{2},0,-\frac{1}{2},1) & (0,\frac{1}{2},-\frac{1}{2},1).  \\
\end{array}
\]
\subsection{Estimates when $1$ or $2$ is the bad index}
Now, the operator can be estimated in nearly the same way, although there are some minor changes which we now describe.  We prove the estimates for $2$ being the bad index.  The case for $1$ being the bad index is completely analogous.

We define the exceptional set
\[
\Omega_C=\bigcup_{j=1}^4\{M(\chi_{E_j})>C|E_j|/|E_2|\},
\]
where again $M$ is the Hardy--Littlewood maximal operator.  For sufficiently large $C$, we can define $E_2'=E_2\backslash\Omega_C$ to get an appropriate major subset.

Now, we make two assumptions of a similar type to the ones we made before: we restrict to tiles $\vec{P}$ with
\[
2^k\le 1+\frac{\textrm{dist}(I_{\vec{P}},\mathbb{R}/\Omega_C)}{|I_{\vec{P}}|}\le 2^{k+1}
\]
and to tiles $\vec{Q}$ with
\[
2^{k'}\le1+\frac{\textrm{dist}(I_{\vec{Q}},\mathbb{R}/\Omega_C)}{|I_{\vec{Q}}|}\le 2^{k'+1},
\]
which is harmless provided we get summability in $k,k'$.  One then proceeds in exactly the same fashion, except that one needs to make the following changes to the $\size_3$ and $\energy_3$:
\[
\size_3((a_{P_2})_{\vec{P}\in\vec{\mathbf{P}}})\lesssim 2^{(-M\theta)k},
\]
where one must use the crude estimate $\int_{E_j}\tilde{\chi}_{I_{\vec{Q}}}^M\le |I_{\vec{Q}}|$.  We are already choosing $M$ depending on the exponent parameters, so the presence of $\theta$ is ok, provided it is nonzero.
We also get 
\[
\energy_3((a_{P_2})_{\vec{P}\in\vec{\mathbf{P}}})\lesssim 2^{-M\theta k'}|E_1|^{(2-\theta)/2}|E_2|^{(\theta-1)/2},
\]
for some $0<\theta<1$.  One gets summability in $k,k'$, so we may ignore their presence.  The estimate one gets as before (ignoring the small factor $1/\mu$) is
\[
\frac{|E_3|^{(1+a_1)/2}|E_4|^{(1+a_2)/2}}{|E_2|^{1-a_3}}\left(|E_1|^{(2-\theta)/2}|E_2|^{(\theta-1)/2}\right)^{1-a_3},
\]
where $0<a_1,a_2,a_3<1$ and $a_1+a_2+a_3=1$.  Now pick
\begin{enumerate}
\item $a_1=2\alpha_3-1$,
\item $a_2=2\alpha_4-1$,
\item $a_3=2(\alpha_1+\alpha_2)+1$, and
\item $\theta=(3\alpha_1+2\alpha_2)/(\alpha_1+\alpha_2)$.
\end{enumerate}
This numerology transforms the previous line to
\[
|E_3|^{\alpha_3}|E_4|^{\alpha_4}|E_1|^{\alpha_1}|E_2|^{\alpha_2}.
\]
i.e. producing a weak-type $\alpha$ estimate.  One may now check that tuples arbitrarily close to the following are available:
\[
\begin{array}{ccc}
(1,-\frac{3}{2},\frac{1}{2},1) & (1,-\frac{3}{2},1,\frac{1}{2}).
\end{array}
\]
By doing the same analysis for 1 being the bad index, one gets
\[
\begin{array}{ccc}
(-\frac{3}{2},1,\frac{1}{2},1) & (-\frac{3}{2},1,1,\frac{1}{2}).
\end{array}
\]
\section{Main Result}

\begin{proof}[Proof of Theorem 1.1]
The above establishes that the 4-linear form given by
\[
\Lambda(f_1,f_2,f_3,f_4)=
\sum_{P\in\vec{\mathbf{P}}}\frac{1}{|I_P|^{1/2}}\langle f_3, \phi_P^1\rangle\left\langle B_P(f_1,f_2), \phi_P^2\right\rangle\left\langle f_41_{|I_P|\ge 2^{N_2(x)}},\phi_P^3\right\rangle,
\]
where
\[
B_P(f_1,f_2):=\sum_{Q\in\vec{\mathbf{Q}}:\omega_{Q_3}\subset \omega_{P_2}}\frac{1}{|I_Q|^{1/2}}\langle f_1, \phi_{Q}^1\rangle\langle f_2, \phi_{Q}^2\rangle\phi_{Q}^3,
\]
where $\vec{\mathbf{P}}$ and $\vec{\mathbf{Q}}$ are finite, rank-1 families of tritiles satisfies restricted weak-type estimates arbitrarily close to the following points twelve points in $\R^4$:
\[
\begin{array}{cccccc} 
(1,-\frac{3}{2},\frac{1}{2},1) & (1,-\frac{3}{2},1,\frac{1}{2})&(-\frac{3}{2},1,\frac{1}{2},1) & (-\frac{3}{2},1,1,\frac{1}{2}),\\
\\
(1,0,-\frac{1}{2},\frac{1}{2}) & ( 1,0,-\frac{1}{2},\frac{1}{2}) & ( \frac{1}{2},0,-\frac{1}{2},1) & (0,\frac{1}{2},-\frac{1}{2},1),  \\
\\
(1,0,\frac{1}{2},-\frac{1}{2}) & (0,1,\frac{1}{2},-\frac{1}{2}) & (\frac{1}{2},0,1,-\frac{1}{2}) & (0,\frac{1}{2},1,-\frac{1}{2}).
\end{array}
\]
We claim that it follows that $\Lambda$ is of restricted weak type $\alpha$ for every $\alpha$ in the interior of the convex hull of these twelve points.  This follows by standard arguments, but we give them fairly explicitly here.  It may be useful to consult \cite[Figure 1]{MTTBiest1} to get a visual of this discussion.  Observe that points arbitrarily close to each of $(1,0,0,0),(0,1,0,0),(0,0,1,0),(0,0,0,1)$ can be written as (strictly) convex linear combinations of four points arbitrarily close to the twelve points listed above; in particular, this can be done so that the four chosen points are negative in pairwise different coordinates.  For example, observe that, for any small $\theta>0$,
\[
\begin{pmatrix}1\\0\\0\\0\end{pmatrix}=\left(\frac{3}{10}-\theta\right)\begin{pmatrix}1\\1/2\\1\\-3/2\end{pmatrix}+\left(\frac{1}{5}+\theta\right)\begin{pmatrix}1\\1/2\\-3/2\\ 1\end{pmatrix}+\left(\frac{1}{2}-5\theta\right)\begin{pmatrix}1\\-1/2\\0\\1/2\end{pmatrix}+\theta\begin{pmatrix}1\\-1/2\\1/2\\0\end{pmatrix}.
\]
One can modify this example slightly to write $(1-\epsilon_1,\epsilon_2,\epsilon_3,\epsilon_4)$ as a convex combination of similarly modified versions of the four vector listed above.  Then by Theorem \ref{goodtuples}, one gets that $\Lambda$ is restricted weak-type at every \emph{good} tuple in the interior of the convex hull of $(1,0,0,0),(0,1,0,0),(0,0,1,0),(0,0,0,1)$.  But then one can write any element of the interior of the convex hull of these twelve points in terms of a (strictly) convex linear combination of two good tuples and two tuples which are bad at the same index and invoke Lemma \ref{badtuples1}.  We stress again that it may be useful to consult \cite[Figure 1]{MTTBiest1}.  Thus $\Lambda$ is actually restricted weak-type everywhere in the interior of the convex hull of these twelve points.  
\end{proof}
The previous theorem, together with the weak-type interpolation result found in Lemma \ref{restrictedweaktostrong}, establish the following theorem theorem.
\begin{theorem}\label{mainresult}
Define $T(f_1,f_2,f_3)$ by
\begin{align*}
\left|\sum_{i\gg j}\int\right.&\left.\int_{\mathbb{R}^3}\hat{f_1}(\xi_1)\hat{f_2}(\xi_2)\hat{f_3}(\xi_3)\theta_i(\xi_2-\xi_1)\right.\\ &\left.\phi_j\left(\xi_3-\frac{\xi_1+\xi_2}{2}\right)e^{2\pi ix(\xi_1+\xi_2+\xi_3)}d\xi1_{j\ge N_2(x)}\right|,
\end{align*}
where $N_2(x)$ is an arbitrary, integer-valued function on $\mathbb{R}$ and $\theta_i$ and $\phi_j$ are defined as they were in earlier sections.  Let $D$ denote the interior of the convex hull in $\{(1/p_1,1/p_2,1/p_3,1/p_4):\sum 1/p_i=1\}$ of the 4-tuples given in the preceding proof.  Suppose that $1<p_1,p_2,p_3\le\infty$ and $1\le p_4'<\infty$ where $1/p_4'=1-1/p_4$ are such that $(1/p_1,1/p_2,1/p_3,1/p_4)$ is in $D$.  Then
\[
T:L^{p_1}\times L^{p_2}\times L^{p_2}\rightarrow L^{p_4'}.
\]
\end{theorem}

\begin{corollary}
Suppose that $T$ is as in the previous theorem.  Then $T:L^2\times L^2\times L^2\rightarrow L^{2/3}$ is bounded.
\end{corollary}
This corollary is of particular interest since we get a strong bound into $L^{2/3}$.  All the ``trivial'' methods of estimation require putting one of the $f_i$ into $L^\infty$ and then using previous methods to make estimations on the remaining objects; however, the only estimates available have either the other $f_i$ in $L^p$ and $L^q$ where either $p^{-1}+q^{-1}=1$ (if $f_2\in L^\infty$ and applying H\"{o}lder on the tensor product of two maximal operators) or $p^{-1}+q^{-1}>3/2$ (either $f_1$ or $f_3$ in $L^\infty$ and applying time-frequency analysis in the spirit of of Lacey's original paper on the maximal bilinear operator, \cite{LBMO}, or the relevant special case of Demeter, Tao, Thiele, \cite{DTT}).  Either way, one cannot produce bounds using the prior estimates so that the target space is actually $L^{2/3}$.

\bibliography{thesis_paper_1}
\bibliographystyle{plain}

\end{document}